\documentclass[11pt, reqno]{amsart}
\usepackage{amsthm,amssymb,amsmath,enumerate,graphicx,psfrag,enumitem,mathrsfs}

\usepackage[margin=2.5cm]{geometry}
\usepackage{hyperref}
\hypersetup{colorlinks=true,
   citecolor=blue,
   filecolor=blue,
   linkcolor=blue,
   urlcolor=blue
}

\usepackage{algorithm}
\usepackage{tikz}

\usetikzlibrary{backgrounds}

\newtheorem{definition}{Definition}[section]
\newtheorem{claim}[definition]{Claim}

\newtheorem{theorem}[definition]{Theorem}

\newtheorem{lemma}[definition]{Lemma}

\numberwithin{equation}{section}

\def \ch{{\rm ch}}
\def \Mad{{\rm Mad}}

\begin{document}
\title{Strong edge-colorings of sparse graphs with large maximum degree}
\author{Ilkyoo Choi}
\address{Department of Mathematical Sciences, KAIST, Daejeon, Republic of Korea.}
\email{ilkyoo@kaist.ac.kr}

\author{Jaehoon Kim}
\address{School of Mathematics, University of Birmingham, 
Edgbaston, Birmingham, B15 2TT, United Kingdom}
\email{j.kim.3@bham.ac.uk}

\author{Alexandr V. Kostochka}
\address{University of Illinois at Urbana--Champaign, Urbana, IL 61801, USA
 and Sobolev Institute of Mathematics, Novosibirsk 630090, Russia}
\email{kostochk@math.uiuc.edu}

\author{Andr\'e Raspaud}
\address{LaBRI (Universit\'e de Bordeaux), 351 cours de la Lib\'eration, 33405 Talence Cedex, France}
\email{andre.raspaud@labri.fr}

\thanks{
The first author is supported by the National Research Foundation of Korea (NRF) grant funded by the Korea government (MSIP) (NRF-2015R1C1A1A02036398).
The second author is supported by the European 
Research Council under the European Union's Seventh Framework Programme (FP/2007--2013) / ERC Grant Agreements no. 306349 (J.~Kim).
The research is conducted while the second author visited University of Illinois Urbana-Champaign as part of BRIDGE Strategic Partnership.
 Research of the third author
is supported in part by NSF grant  DMS-1600592
and  by
grants 15-01-05867 and 16-01-00499  of the Russian Foundation for Basic Research.
Research of the fourth author received
financial support from the French 
State, managed by the
French National Research Agency (ANR) within the
  "Investments for the future" Programme IdEx Bordeaux - CPU 
(ANR-10-IDEX-03-02).
}
\maketitle

\begin{abstract}
A {\em strong $k$-edge-coloring} of a graph $G$ is a mapping from
$E(G)$ to $\{1,2,\ldots,k\}$ such that every two adjacent edges or two
edges adjacent to the same edge receive  distinct colors. 
The {\em strong chromatic index} $\chi_s'(G)$ of a graph $G$ is the smallest integer $k$ such that $G$ admits a
strong $k$-edge-coloring.
We give bounds on  $\chi_s'(G)$  in terms of the maximum degree $\Delta(G)$ of a graph $G$
when $G$ is sparse, namely, when $G$ is $2$-degenerate or when the maximum average degree $\Mad(G)$ is small. 
We prove that the strong chromatic index of each $2$-degenerate graph $G$ is at most $5\Delta(G) +1$.
Furthermore, we show that for a graph $G$, if $\Mad(G)< 8/3$ and  $\Delta(G)\geq 9$, then $\chi_s'(G)\leq 3\Delta(G) -3$ (the bound $3\Delta(G) -3$ is sharp) and 
if $\Mad(G)<3$ and  $\Delta(G)\geq 7$, then $\chi_s'(G)\leq 3\Delta(G)$ (the restriction $\Mad(G)<3$ is sharp). 
\end{abstract}

\section{Introduction}\label{Introduction}

A {\em strong $k$-edge-coloring} 
 of a graph $G$ is a mapping from
$E(G)$ to $\{1,2,\ldots,k\}$ such that every two adjacent edges or two
edges adjacent to the same edge receive  distinct colors. In other
words, the graph induced by each color class is an induced matching.
 The {\em strong chromatic index} of $G$, denoted by
$\chi_s'(G)$, is the smallest integer $k$ such that $G$ admits a
  strong $k$-edge-coloring.

\medskip
Strong edge-coloring was introduced by Fouquet and Jolivet~\cite{FJ83,fj1984} and was used to solve the frequency assignment
problem in some radio networks.
 For more details on applications  see
\cite{BI+06,NK+00,Ram97,RL93}.

\medskip
An obvious upper bound on $\chi_s'(G)$ (given by a greedy coloring) is
$2\Delta(G)(\Delta(G)-1)+1$ where $\Delta(G)$ denotes the maximum
degree of $G$.
 Erd\H{o}s and Ne\v set\v ril \cite{e1988,en1989}
conjectured that for every graph $G$
 with maximum degree $\Delta$, 
$$
\chi_s'(G)\leq
\begin{cases}
\frac{5}{4}\Delta^2 &\mbox{if $\Delta$ is even}\\
\frac{5}{4}\Delta^2-\frac{\Delta}{2}+\frac{1}{4} &\mbox{if $\Delta$ is odd}
\end{cases}
$$
The bounds in the conjecture are sharp, if the conjecture is true.

\medskip
The first nontrivial upper bound on $\chi_s'(G)$ was given by Molloy
and Reed~\cite{MR97}, who showed that  $\chi_s'(G)\le 1.998 \Delta^2$,
if $\Delta$ is sufficiently large.
The coefficient $1.998$ was improved to $1.93$ (again, for
sufficiently large  $\Delta$) by Bruhn and Joos~\cite{BJ2015}. Recently, 
 Bonamy,  Perrett and Postle~\cite{BPP} announced an even better coefficient of~$1.835$.
For $\Delta=3$, the conjecture was settled independently by
Andersen~\cite{And92} and by Hor\'ak, Qing and~Trotter~\cite{HQT93}.
Cranston~\cite{Cra06} proved that
every graph with $\Delta \le 4$ admits a strong edge-coloring
with $22$ colors, which is $2$ more than the conjectured bound.


The strong chromatic index was studied  for various families of
graphs, such as cycles, trees, $d$-dimensional cubes, chordal graphs,
and Kneser graphs, see \cite{Mah00}. There was also a series of 
papers~\cite{FGST,HMRV2013,KLRSWY2016}
on strong edge-coloring planar graphs.
In particular,
Faudree,  Gy\'{a}rf\'{a}s, Schelp and Tuza~\cite{FGST} proved that 
$\chi_s'(G)\leq 4\Delta+4 $ for every planar graph $G$  
with maximum degree $\Delta$ and
 exhibited, for every
  integer $\Delta \geq 2$, a planar graph with maximum degree $\Delta$
  and strong chromatic index $4\Delta -4$. 
Borodin and Ivanova~\cite{BI13} showed that every planar graph $G$ with maximum degree $\Delta\geq 3$ and
girth $g\geq 40\lfloor \Delta/2\rfloor$ satisfies $\chi_s'(G)\leq 2\Delta-1$, and that the bound $2\Delta-1$ is sharp.
Chang,  Montassier,  Pecher and Raspaud~\cite{CMPR} relaxed the restriction on $g$ to $g\geq 10\Delta+46$ for
$\Delta\geq 4$. 

Hud\' ak,  Lu\v zar,  Sot\' ak and  \v Skrekovski~\cite{HLSS} proved that $\chi_s'(G)\leq 3\Delta+6$
for every planar graph $G$ with maximum degree $\Delta\geq 3$ and
girth $g\geq 6$. Recently, Bensmail,  Harutyunyan,  Hocquard and  Valicov~\cite{BHHV} improved the upper
bound $3\Delta+6$ for such graphs to $3\Delta+1$. With the stronger restriction of $g\geq 7$,
Ruksasakchai and Wang~\cite{RW} reduced the bound $3\Delta+1$ to $3\Delta$.

Clearly, planar graphs with large girth are sparse.
The problem of strong edge-coloring was also studied for general sparse graphs. A natural measure of sparsity is {\em degeneracy}: 
a graph $G$ is $d$-{\em degenerate} if every subgraph $G'$ of $G$
has a vertex of degree at most $k$ (in $G'$). Chang and Narayanan~\cite{CN13} proved that
$\chi_s'(G)\leq 10\Delta-10$ for every $2$-degenerate graph $G$ with maximum degree $\Delta\geq 2$. 
Luo and Yu~\cite{LY13} improved the bound $10\Delta-10$ to $8\Delta-4$. 
A more general bound by Yu~\cite{GY} allowed to reduce the bound for $2$-degenerate graphs to $6\Delta-5$, and Wang~\cite{TW} improved it to  $6\Delta-7$.

In this paper, we prove three bounds on the strong chromatic index of sparse graphs in terms of the maximum degree. 
Two of our bounds yield new bounds for planar graphs with  girths $6$ and $8$.

Our first result is on $2$-degenerate graphs. It improves the aforementioned bounds in~\cite{CN13,LY13,TW} for $\Delta\geq 9$.
\begin{theorem}\label{2deg}
Every $2$-degenerate graph $G$ with maximum  degree $\Delta$ satisfies $\chi'_s(G)\leq 5\Delta+1$. 
\end{theorem}

A finer measure of sparsity is
the {\em maximum average degree}, denoted $\Mad(G)$, which is the maximum of $2\frac{|E(G')|}{|V(G')|}$ over all
nontrivial subgraphs $G'$ of a graph $G$. By definition, $\Mad(G)<4$ for every $2$-degenerate graph $G$.
Two of our  results show that if $\Mad(G)<3$, then
we can use significantly fewer than $5\Delta$ colors.
The graphs $K_{\Delta}(t)$ defined below
show that  our bounds are almost optimal.
Let $K_{\Delta}(t)$ be the graph obtained from $K_t$ by adding $\Delta-t+1$ pendant edges to each vertex in $K_t$. 
It is easy to check that $\Mad(K(t))=t-1$ and $\chi'_s(K_{\Delta}(t))=|E(K_{\Delta}(t))|=t\Delta-{t\choose 2}$.
In particular, 
\begin{itemize}
\item $\Mad(K_{\Delta}(2))=1$ and $\chi'_s(K_{\Delta}(2))=2\Delta-1$,
\item $\Mad(K_{\Delta}(3))=2$ and $\chi'_s(K_{\Delta}(3))=3\Delta-3$, 
\item $\Mad(K_{\Delta}(4))=3$ and $\chi'_s(K_{\Delta}(4))=4\Delta-6$.
\end{itemize}

Our second result is: 
\begin{theorem}\label{8:3}
Let $\Delta \geq 9$ be an integer.
Every graph $G$ with maximum average degree less than $8/3$ and maximum degree at most $\Delta$ satisfies $\chi'_s(G)\leq 3\Delta-3$.
\end{theorem} 

The graph $K_{\Delta}(3)$ above shows that the bound $3\Delta-3$ is best possible. 
The graph $K'_{\Delta}(4)$ defined below shows that the bound on 
the maximum average degree is close to optimal:\\
We start from a copy $R$ of $K_4$ with vertex set $\{v_1,v_2, v_3,v_4\}$ and
 let $K'_{\Delta}(4)$ be the graph  obtained from $R$  by subdividing  the edge  $v_3v_4$ with a vertex $u$
and then adding $\Delta-3$ pendant edges to each of $v_1,v_2,v_3$. 
It is not hard to check that the maximum degree of $K'_{\Delta}(4)$ is $\Delta$,
  $\Mad(K'_{\Delta}(4))= 14/5$, and $\chi'_s(K'_{\Delta}(4)) = 3\Delta -2$.

Our last result is: 
\begin{theorem}\label{mad3}
Let $\Delta \geq 7$ be an integer.\footnote{We can prove the result for  $\Delta\geq 6$, but that would make the proof longer and more complicated.} 
Every graph $G$ with maximum average degree less than $3$ and maximum degree at most $\Delta$ satisfies $\chi'_s(G)\leq 3\Delta$.
\end{theorem} 

Note that  for small $\Delta$, namely for $\Delta\leq 4$, the slightly weaker bound  of $3\Delta+1$ was proved by
Ruksasakchai and Wang~\cite{RW}.
Since  $\Mad(K_{\Delta}(4))=3$ and $\chi'_s(K_{\Delta}(4))=4\Delta-6$,
 the restriction on the maximum average degree in Theorem~\ref{mad3}  is best possible for $\Delta\geq 7$.
 The graph $K'_{\Delta}(4)$ above with $\Mad(K'_{\Delta}(4))= 14/5$ and $\chi'_s(K'_{\Delta}(4)) = 3\Delta -2$
shows that the bound $3\Delta$ is also close to the best possible.

Since $\Mad(G)<\frac{2g}{g-2}$ for every planar graph $G$ with girth $g$, Theorem~\ref{8:3} yields that 
 $\chi_s'(G)\leq 3\Delta-3$
for every planar graph $G$ with maximum degree $\Delta\geq 9$ and
girth $g\geq 8$ and Theorem~\ref{mad3} implies that
 $\chi_s'(G)\leq 3\Delta$
for every planar graph $G$ with maximum degree $\Delta\geq 7$ and
girth $g\geq 6$. The last result improves the  bounds in~\cite{BHHV,HLSS,RW} mentioned above for $\Delta\geq 7$.

The structure of the paper is as follows. In the next section   we introduce some notation and  prove useful lemmas. In Section 3, 4 and 5, we prove Theorem \ref{2deg}, Theorem \ref{8:3} and Theorem \ref{mad3}, respectively. 

\section{Notation and preliminaries }
Let $[k]:=\{1,\dots, k\}$. For a function $f$ defined on a set $A'$ with $A'\subseteq A$, we denote $f(A):= f(A')=\{ f(a) : a \in A' \}$.  For a graph $G$, let $\overline{d}(G)$ 
be the average degree of $G$. We define $\Mad(G):=\max_{H\subseteq G} \overline{d}(H)$.
A vertex $v\in V(G)$ is a $d^+_G$-{\em vertex}  if $d_G(v)\geq d$. If $G$ is clear from the context, then we simply say that $v$ is a $d^+$-{\em vertex}.
A $d^+_G$-{\em neighbor} of a vertex
$v\in V(G)$ is a neighbor of $v$ that is a $d^+_G$-vertex.
A $d^-_G$-{\em vertex}, a  $d_G$-{\em vertex}, a $d^-_G$-{\em neighbor}, and a  $d_G$-{\em neighbor} are defined similarly.

For a vertex $v\in V(G)$, let $N_{G}(v)$ denote the set of all neighbors of $v$ in $G$ and let $\Gamma_{G}(v)$ denote
the set of all edges incident to $v$ in $G$. For an edge $e=uv$, let $$N_{G}[e]:= \Gamma_{G}(u)\cup \Gamma_{G}(v) \text{ and } N_{G}^2[e]:= \bigcup_{w \in N_{G}(u)\cup N_{G}(v)} \Gamma_{G}(w).$$ 
A function $f:E(G)\rightarrow [k]$ is a {\em strong $k$-edge-coloring} of $G$ if  $f(e)\neq f(e')$ for any $e,e'\in E(G)$ with $e' \in N_{G}^2[e]\setminus\{e\}$. 
Since below we only consider strong edge-colorings, for brevity we will simply call them  {\em colorings}.
A function $f : E'\rightarrow [k]$ is a {\em partial $k$-coloring} of $G$ on $E'$ if $E'\subseteq E(G)$ and $f(e)\neq f(e')$ for any $e,e'\in E'$ 
with $e' \in N_{G}^2[e]\setminus\{e\}$. 
For a  partial $k$-coloring $f:E'\rightarrow[k]$ of a graph $G$ and $e\in E(G)$, let the {\em $f$-multiplicity of $e$}, $m(f,e)$, 
  be $m(f,e):= |N_{G}^2[e]\cap E'| - |f(N_{G}^2[e])|$. 
  Note that $m(f,e)$ counts multiple occurrences of  all colors in $N_{G}^2[e]$. 
  By definition, $|f(N_{G}^2[e])| = |N_{G}^2[e]\cap E'| - m(f,e)$. In particular, if $N_{G}^2[e]$ contains two edges with the same color, then $m(f,e)\geq 1$.

For a partial $k$-coloring $f : E \rightarrow [k]$ and $e',e''\in E$, we often say ``we extend $f$ to $e'$ by coloring it with a color $\alpha$''. This means that we replace $f$ with a new function $f':E \cup\{e'\} \rightarrow [k]$ such that $f'(e):= f(e)$ for all $e\in E\setminus\{e'\}$ and $f'(e'):= \alpha$. 
Also, we say ``we switch the colors of $e$ and $e'$'' when we replace $f$ with a new function $f':E\rightarrow [k]$ such that $f'(e):= f(e)$ for all $e\in E\setminus\{e',e''\}$, $f'(e'):=f(e'')$ and $f'(e''):= f(e')$. In both cases, we will slightly abuse the notation by denoting the new updated function by $f$.

For a partial $k$-coloring $f$ of $G$, we say that a sequence $(E_1,E_2,\dots, E_s)$ of pairwise disjoint subsets of $E(G)$ is an {\em $(f,k)$-degenerate sequence for $G$} if the following holds:
\begin{itemize}
\item $f : E(G)\setminus( \bigcup_{i=1}^{s} E_i) \rightarrow [k]$ is a partial $k$-coloring of $G$.
\item For every $i\in[s]$ and $e\in E_i$, $|N_{G}^2[e] \setminus \bigcup_{j=i+1}^{s} E_j|\leq k + m(f,e)$.
\end{itemize}

Note that if $(E_1,E_2,\dots, E_s)$ is an $(f,k)$-degenerate sequence for $G$, then the domain of $f$ is exactly $ E(G)\setminus( \bigcup_{i=1}^{s} E_i)$, thus $\bigcup_{i=1}^{s} E_i$ is exactly the set of all edges of $G$ uncolored by $f$.
For a partial $k$-coloring $f$ of a graph $G$,  the graph $G$ is {\em $(f,k)$-degenerate} if there exists an 
$(f,k)$-degenerate sequence $(E_1,\dots, E_s)$ for $G$. If $E_{i} = \{e_i\}$, then for simplicity, instead of $(E_1,\dots, E_s)$, we 
write $(E_1,\dots, E_{i-1}, e_{i}, E_{i+1}, \dots, E_s)$.

The following lemma regarding degeneracy is useful.

\begin{lemma}\label{lem: degenerate}
If a graph $G$ has a partial $k$-coloring $f$ and is $(f,k)$-degenerate, then $\chi'_{s}(G)\leq k$.
\end{lemma}
\begin{proof}
Assume we have a partial $k$-coloring $f$ of $G$ with domain $E_0$ and $(E_1,\dots, E_s)$ is an $(f,k)$-degenerate sequence on $G$. 
Let $S=(e_1,\dots, e_t)$ be an ordering of all edges in $\bigcup_{i=1}^{s} E_i$ such that for $j\in[s-1]$, all edges in $E_j$ come before any edge in $E_{j+1}$.

We iteratively color edges in $S$ in order to extend $f$ to $f_1,\dots, f_t$. Assume that we have colored $e_1,\dots, e_{i-1}$ for $i\in [t]$ and
have a partial $k$-coloring $f_{i-1}$ on $E_0\cup \{e_1,\dots, e_{i-1}\}$. Let $e_i\in E_j$.
Note that $m(f,e)\leq m(f_{i-1},e)$ for any $e\in E(G)$.
Then since $(E_1,\dots, E_s)$ is an $(f,k)$-degenerate sequence for $G$,  
\begin{align*}
\left|f_{i-1}(N_{G}^2[e_i])\right| 
&= \left|N_{G}^2[e_i] \setminus \{e_i,\dots, e_{t}\}\right|- m(f_{i-1},e_i) 
\leq  \left|N_{G}^2[e_i] \setminus \left(\{e_i\}\cup \bigcup_{\ell=j+1}^s E_{\ell}\right) \right| - m(f,e_i) \\
& =\left|N^2_{G}[e_i]\setminus \bigcup_{\ell=j+1}^sE_\ell\right| -|\{e_i\}| -m(f,e_i) 
\leq k-1.
\end{align*}
Thus we can choose a color $c \in [k] \setminus f_{i-1}(N_{G}^2[e_i])$.
Let 
$$f_i(e) := \left\{ \begin{array}{ll}
f_{i-1}(e) & \text{ if } e\in E_0\cup \{e_1,\dots, e_{i-1}\},\\
c & \text{ if } e=e_i.
 \end{array}\right.$$
Now, $f_{i}$ is  a partial $k$-coloring of $G$ with domain $E_0\cup \{e_1,\dots, e_i\}$. 
By repeating this process, we get a strong $k$-edge-coloring of $G$. 
Thus $\chi'_{s}(G)\leq k$.
\end{proof}

We say a vertex $u$ is {\em pale in $G$} if 
$u$ has at most two $3^+_G$-neighbors.
A vertex $u$ is {\em light in $G$} if all vertices in $N_{G}(u)$ except at most two are pale. Since a vertex with degree two is  pale, each pale vertex is also light.

\begin{lemma}\label{lem: 3-reg light}
If $G'$ is a  subgraph of a graph $G$ with $\delta(G')\geq 3$, then $G'$ has no vertex that is light in $G$.
\end{lemma}
\begin{proof} 
Assume   $G$ contains a subgraph $G'$ with $\delta(G')\geq 3$ and $v\in V(G')$ is light in $G$. 
Then 
 $|N_{G'}(v)|\geq 3$. So by the definition of ``light", there is a pale $w\in N_{G'}(v)$. Similarly, 
 $|N_{G'}(w)|\geq 3$. So by the definition of ``pale", some $u\in N_{G'}(w)$ is a $2^-_G$-vertex,
 contradicting $\delta(G')\geq 3$.
\end{proof}

\begin{lemma}\label{lem: 2 degen}
Every $2$-degenerate graph $G$ with  $\Delta(G)\geq 3$ has a vertex $v$ of degree at least three 
that is adjacent to at most two vertices of degree at least three.
\end{lemma}
\begin{proof} Let $V'$ denote the set of $3^+$-vertices. Since $\Delta(G)\geq 3$, we know that $V'\neq \emptyset$. Since $G$ is $2$-degenerate,
$G[V']$ is also $2$-degenerate. In particular, $G[V']$ has a vertex $v$ of degree at most $2$. This $v$ satisfies the lemma.
\end{proof}

\section{$2$-degenerate graphs}
In this section we prove the following stronger result, which implies Theorem~\ref{2deg}.

\begin{theorem}\label{lem: 2-degenerate}
Let  a $2$-degenerate graph $G$ and a positive integer $D\geq 2$ satisfy the following:
\begin{itemize}
\item[(A1)] $\Delta(G)\leq D +2$.
\item[(A2)] For $t\in [2]$, every vertex $v\in V(G)$ with $d_{G}(v)= D+t$ is adjacent to at least $t$ vertices of degree one.
\end{itemize}
Then $\chi'_s(G)\leq 5D+1$.
\end{theorem}
\begin{proof} 
Let $G$ be a counterexample to the statement with the fewest $2^+$-vertices, and subject to this,  with the fewest edges.
Let $G^*$ be obtained from $G$ by deleting all vertices of degree $1$ in $G$. By (A1) and (A2), $\Delta(G^*)\leq D$.

First we prove that 
\begin{equation}\label{eq:uv}
\text{if $v\in V(G^*)$ and $d_{G^*}(v)\leq 2$,  then $d_{G^*}(v)=d_G(v)$.}
\end{equation}
Indeed, suppose  to the contrary that $N_{G^*}(v) = \{w_1,\dots, w_{t}\}$ where $t\in [2]$ and there is $u\in N_G(v) \setminus N_{G^*}(v)$ with $d_G(u) = 1$.
The graph $G':= G- u$  is  $2$-degenerate and has  no more $2^+_{G'}$-vertices  than $G$ has $2^+_{G}$-vertices. $G'$ also satisfies (A1) and (A2).
Furthermore, $G'$ has   strictly fewer edges than  $G$. 
By the minimality of $G$, the graph $G'$ has a $(5D+1)$-coloring $f$. Since
$$|N^2_G[uv]|\leq \left|\Gamma_{G}(v)\cup\bigcup_{i=1}^{t}\Gamma_{G}(w_i)\right| \leq D+2+D+2+ D+2-2 \leq 5D+1,$$
when $D\geq 2$, we can extend $f$ to $uv$, a contradiction. This proves~\eqref{eq:uv}. In particular,~\eqref{eq:uv} yields
\begin{equation}\label{j26}
\delta(G^*)\geq 2.
\end{equation}

If $\Delta(G^*)\leq 2$, then by~\eqref{j26},  $G^*$ is a disjoint union of cycles. Then by~\eqref{eq:uv}, $G$ itself is a disjoint union of cycles.
So by the minimality of $G$, it is a cycle. Since each edge of a cycle has at most four edges at distance at most $2$, it is
strong $5$-edge-colorable.
This   contradicts the choice of $G$ since $5\leq 5D+1$. Thus $\Delta(G^*)>2$. Then by Lemma~\ref{lem: 2 degen},
 $G^*$ has a vertex $v$ with $d_{G^*}(v)\geq 3$  that  is adjacent to at most two $3^+_{G^*}$-vertices.  We fix such a vertex to be $v$ and let $N_{G^*}(v) =\{ v_1, v_2, u_1,\dots, u_t\}$ where $d_{G^*}(u_i)=2$ for all $i\in [t]$. (It could be that $d_{G^*}(v_1)= 2$ and/or $d_{G^*}(v_2)= 2$.)

By the choice of $v$, 
we know $t\geq 1$.  By~\eqref{eq:uv} and~\eqref{j26}, $d_{G^*}(u_i)=d_{G}(u_i)=2$ for each $i\in [t]$. For each $i\in[t]$, let $N_{G}(u_i) = \{v, w_i\}$.
By (A1) and (A2), we have
\begin{align}\label{eq: t+s and t}
 t+2\leq D. 
 \end{align}
 We claim that 
 \begin{align}\label{eq: v deg same}
\text{$d_{G^*}(v) = d_{G}(v)$.}
\end{align}
 Indeed, suppose to the contrary that $u\in N_G(v) \setminus N_{G^*}(v)$ with $d_G(u) = 1$.
The graph $G'':= G- u$  is  $2$-degenerate and has  no more $2^+_{G''}$-vertices  than $G$ has $2^+_{G}$-vertices. $G''$ also satisfies (A1) and (A2).
Furthermore, $G''$ has   strictly fewer edges than  $G$. 
By the minimality of $G$, the graph $G''$ has a $(5D+1)$-coloring $f$. Since
$$|N^2_G[uv]|\leq d_{G}(v) + \sum_{i=1}^{2} (d_{G}(v_i)-1) + \sum_{i=1}^{t} (d_{G}(u_i)-1) \stackrel{\eqref{eq:uv}}{\leq} (D+2)+2(D+1) + t \stackrel{\eqref{eq: t+s and t}}{\leq}  5D+1,$$
when $D\geq 1$, we can extend $f$ to $uv$, a contradiction. This proves~\eqref{eq: v deg same}.

Suppose $w_1$ has exactly $h$ neighbors of degree $1$ in $G$. Let $H$ be the graph obtained from $G-vu_1$ by 
adding $\ell:=\max\{0,2-h\}$ new vertices $x_1,\ldots,x_\ell$, each of which is adjacent only to $w_1$.
Since the degree of $u_1$ in $G$ is $2$ and  in $H$ is $1$,  $H$ has fewer $2^+$-vertices than $G$. Also
by construction,
\begin{equation}\label{j28}
\mbox{
$w_1$ has at least $3$ neighbors of degree $1$ in $H$, say $u_1,u'_1$, and $u''_1$.}
\end{equation}

It is not hard to check that $H$ inherits  properties (A1) and (A2) from $G$.
Note that $H$ has fewer $2^+$-vertices than $G$. 
So, by the minimality of $G$, the $2$-degenerate graph $H$ has
 a $(5D+1)$-coloring $f$. By~\eqref{j28}, we can switch the colors of $w_1u_1,w_1u'_1$, and $w_1u''_1$ so that
 \begin{equation}\label{j31}
 f(w_1u_1)\notin \{f(vv_1),f(vv_2)\}.
\end{equation}

\noindent {\bf\em Case 1.} $f(w_1u_1)\notin \{f(vu_2),\ldots,f(vu_t)\}$. Together with~\eqref{eq: v deg same} and \eqref{j31}, this
yields that $f\vert_{E(G)}$ 
is a partial coloring of $G$ where only $vu_1$ is not colored.
Since
\begin{align}\label{j312}
|N^2_{G}[vu_1] | &\leq d_{G}(v_1)+ d_{G}(v_2)+ (d_{G}(w_1)-1) + \sum_{i=1}^{t} d_{G}(u_i) \\ \nonumber
& \leq  (D+2) + (D+2) +  (D+1)+2t \stackrel{\eqref{eq: t+s and t}}{\leq} 3D+5 + 2D-4 \leq 5D+1,
\end{align}
we can extend $f$ to $vu_1$.
\newline

\noindent {\bf\em Case 2.} There is $i\in [t]\setminus\{1\}$ such that $f(w_1u_1)=f(vu_i)$. Then by~\eqref{eq: v deg same} and \eqref{j312},
we can choose $f(vu_1)$ so that the only conflict in $f$ will be that $f(w_1u_1)=f(vu_i)$. Let $f'$ be obtained from $f$ by uncoloring
$vu_i$. Then we have Case 1 with $f'$ in place of $f$ and $u_i$ in place of $u_1$. 
This proves the theorem.
\end{proof}

\section{Graphs with maximum average degree less than $8/3$}

In this section we prove Theorem \ref{8:3}: 

{\em If $\Delta\geq 9$ and $G$ is a graph with $\Mad(G)<8/3$ and $\Delta(G)\leq \Delta$, then $\chi'_s(G)\leq 3\Delta(G)-3$. }

\bigskip Similarly to the proof of Theorem \ref{lem: 2-degenerate},
consider a counterexample $G$  with the fewest $2^+$-vertices, and subject to this with the fewest edges.
Let $G^*$ be the graph  obtained from $G$ by deleting all vertices of degree $1$. 
Let $\Delta=\Delta(G)$.

\begin{claim}\label{cl: 8/3 min deg 2}
$\delta(G^*)\geq 2$. 
\end{claim}
\begin{proof}
Suppose $u$ is a $1_{G^*}$-vertex where $w$ is the neighbor of $u$ in $G^*$.
Then there exists $v \in N_{G}(u)-w$ with $d_{G}(v)=1$. By the minimality of $G$, the graph $G- v$
has a $(3\Delta-3)$-coloring $f$. Then $f$ is a partial $(3\Delta-3)$-coloring of $G$, and since
\begin{equation}\label{ea1}
|N^2_{G}[uv]| = \sum_{x\in N_{G}(u)} d_{G}(x) \leq d_{G}(w) +  |N_{G}(u)\setminus N_{G*}(u)|\leq  \Delta + (\Delta-1)\leq 3\Delta-3,
\end{equation}
we can extend $f$ to $uv$, a contradiction.
\end{proof}

Say that a vertex $v$ is {\em special} if $d_{G^*}(v)=2$ and $v$ is adjacent to a $(\Delta-1)^+_{G^*}$-vertex.

\begin{claim}\label{cl: special 1}
If $u_1$ and $u_2$ are two adjacent $2_{G^*}$-vertices, then both $u_1$ and $u_2$ are special.
\end{claim}
\begin{proof}
Let $N_{G^*}(u_1)= \{w_1,u_2\}$ and $N_{G^*}(u_2)=\{u_1,w_2\}$. 
Assume $u_1$ is not special, so that $d_{G^*}(w_1)\leq \Delta-2$.

Obtain the graph $H$ from $G$ by first deleting all  $1_G$-neighbors of $u_1$ and $u_2$, then deleting the edge $u_1u_2$, and then adding leaves  adjacent to $w_1$ and $w_2$ so that $d_{H}(w_i)=\Delta$ for $i\in [2]$. 
By construction, $H$ has fewer vertices of degree at least $2$ than $G$. Also, either $\Mad(G)<2$ and hence $\Mad(H)<2$ or $\Mad(H)\leq \Mad(G)$ and hence
$\Mad(H)< 8/3$. So  $H$ has a $(3\Delta-3)$-coloring $f$ by the minimality of $G$. 
Also $w_1$ is incident to at least three pendant edges in $H$, one of which is $u_1w_1$. 
Let $w_1v_1$ and $w_1v_2$ be two other such edges.

By the definition of $H$,
 $|f(\Gamma_{H}(w_1))|= |f(\Gamma_{H}(w_2))| = \Delta$. We define a special color $\alpha$ as follows. If
$f(w_2u_2)\in f(\Gamma_{H}(w_1))$, then let $\alpha:=f(w_2u_2)$. Otherwise,
$f(\Gamma_{H}(w_1))\setminus f(\Gamma_{H}(w_{2}))\neq \emptyset$, and we let $\alpha$ be any color in this difference.
  After this, we switch the colors of $u_1w_1$ and $w_1v_1$, if necessary,  so that $f(u_1w_1) \neq\alpha.$ 
In particular, this means  $f(u_1w_1)\neq f(u_2w_2)$, $f\vert_{E(G)}$ is a partial coloring of $G$. 
  
 We  extend $f$ to $u_1u_2$ by coloring it with a color not in $f(\Gamma_{H}(w_1)\cup \Gamma_{H}(w_2))$, which is possible since $2\Delta+1\leq 3\Delta-3$. Now we will color the pendant edges of $G$ incident with $u_2$, if there are any. In particular, if there is
at least one such edge $u_2u'_2$ and $\alpha\neq f(w_2u_2)$, then we start by letting $f(u_2u'_2)=\alpha$. 
After coloring all pendant edges incident to $u_2$, we have 
\begin{equation}\label{0830}
 \mbox{\em $d_G(u_2)=2$ or $N^2_{G}[u_1z]$ has two edges of color $\alpha$ for each $1_G$-neighbor $z$ of $u_1$.}
\end{equation}
Then we color the remaining edges in $\Gamma_G(u_2)$
one by one, since the only colored edges that could be in conflict are in $\Gamma_G(w_2)\cup \Gamma_G(u_2)\cup \{w_1u_1\}$,
and there are at most $2\Delta+1$ such edges.
Finally, we remove the edges in $E(H)-E(G)$ and  color the pendant edges of $G$ incident with $u_1$ one by one.
For every  pendant edge $u_1z\in E(G)$ incident with $u_1$, at the moment of coloring $u_1z$,
 the number $M(u_1z)$   of the   colors forbidden for $u_1z$ is at most $|N^2_{G}[u_1z]\setminus \{u_1z\}|$. Moreover,
by~\eqref{0830}, if $d_G(u_2)\geq 3$, then  $M(u_1z)\leq |N^2_{G}[u_1z]\setminus\{u_1z\}|-1$. In any case,
$$M(u_1z)\leq d_{G}(w_1)+\max\{2, d_{G}(u_2)-1\}+ (d_{G}(u_1)-3)\leq \Delta+\max\{2,\Delta-1\}+(\Delta-3)=3\Delta-4.$$
So we can always  find a free color for $u_1z$.
This yields a $(3\Delta-3)$-coloring of $G$, a contradiction.
This shows that $u_1$ is special, and by symmetry, this also shows that $u_2$ is special.
\end{proof}

\begin{claim}\label{cl: special 2}
If $u$ is a $3_{G^*}$-vertex with two  $2_{G^*}$-neighbors $v_1$ and $v_2$, then at least one of $v_1,v_2 $ is special.
\end{claim}
\begin{proof}
Let $N_{G^*}(v_i) = \{u, v'_i\}$ for $i\in[2]$ and let $N_{G^*}(u)= \{w,v_1,v_2\}$. 
Suppose to the contrary that both $v_1, v_2$ are not special, so $v'_1, v'_2$ are both $(\Delta-2)^-_{G^*}$-vertices.
We construct a graph $H$ from $G\setminus \{uv_1,uv_2\}$ by deleting all $1_G$-neighbors of $u,v_1$, and $v_2$ and then adding leaves to $v'_1,v'_2$, and $w$ to make the degrees of $v'_1,v'_2$, and $w$ equal to $\Delta$. Since $\Mad(G)<8/3$, and we were adding only $1$-vertices, $\Mad(H)<8/3$.
Since $u,v_1$, and $v_2$ are leaves in $H$ but not in $G$, the graph $H$ has  fewer $2^+$-vertices than $G$.
 Thus by the minimality of $G$, the graph $H$ has a $(3\Delta-3)$-coloring $f$.

For $i\in [2]$, since $v'_i$ is a $(\Delta-2)^-_{G^*}$-vertex, it is adjacent to at least three leaves  in $H$, including $v_i$; let $v'_{i,1}$ and $v'_{i,2}$ be two other leaves adjacent to $v'_i$ in $H$. 

We now define special colors $\alpha_1$ and $\alpha_2$ as follows.
 For  $i\in [2]$, if $f(uw)\in f(\Gamma_{H}(v'_i))$, 
then we let $\alpha_i:=f(uw)$. Otherwise, the set $f(\Gamma_{H}(v'_i)) \setminus f(\Gamma_{H}(w))$ is nonempty, and we let $\alpha_i$ be any color in this difference.
Note that $\alpha_2=\alpha_1$ is possible. By definition,
\begin{equation}\label{0829}
 |\{\alpha_1,\alpha_2,f(uw)\}\cap f(\Gamma_{H}(v'_i))|\leq 2\quad \mbox{for}\; i\in [2].
\end{equation}

Also, for $i\in [2]$, the colors $f(v'_iv_i), f(v'_iv'_{i,1}), f(v'_iv'_{i,2})$ are all distinct. So by~\eqref{0829},
 at least one of them is not in $\{\alpha_1,\alpha_2, f(uw)\}$. Hence we can switch these  colors so that 
$f(v'_iv_i)\notin \{\alpha_1,\alpha_2,f(uw)\}$. Then $f\vert_{E(G)}$ is a partial $(3\Delta-3)$-coloring of $G$.

Let $H'=H+v_1u+v_2u$. Then 
$$|N^2_{H'}[uv_i]|\leq |\Gamma_{H'}(w)|+|\Gamma_{H'}(v'_i)|+|\Gamma_{H'}(v_{3-i})|+1\leq\Delta+\Delta+2+1\leq 3\Delta-3.
$$
Thus for $i\in[2]$, we can  extend $f$ to $uv_i$  by coloring it with a color not in  $f(N^2_{H'}[uv_i])$.

 Now we will color the pendant edges of $G$ incident with $u$, if there are any. 
At first, if there are at least two such edges $uu'_1,uu'_2$ then we do the following :  if $\alpha_1\neq f(wu)$, then we let $f(uu'_1)=\alpha_1$, if $\alpha_2\notin \{f(uw),\alpha_1\}$, then we let $f(uu'_2)=\alpha_2$.
Then we color the remaining pendant edges of $G$ incident with $u$, if there are any. 
After coloring all pendant edges incident to $u$, for $i\in [2]$,
\begin{equation}\label{0830a}
 \mbox{\em $d_G(u)\leq 4$ or $N^2_{G}[v_iz]$ contains two edges of color $\alpha_i$ for every $1_G$-neighbor $z$ of $v_i$.}
\end{equation}
Then we color the remaining edges in $\Gamma_G(u)$
one by one, since the only colored edges that could be in conflict are in $\Gamma_G(w)\cup \Gamma_G(u)\cup \{v_1v'_1,v_2v'_2\}$,
and there are at most $2\Delta+1$ such vertices.
Finally, we remove the edges in $E(H)\setminus E(G)$ and for $i\in [2]$  color the leaves of $G$ incident with $v_i$ one by one.
For every  leaf $v_iz\in E(G)$ incident with $v_i$, at the moment of coloring $v_iz$,
 the number $M(v_iz)$   of the  colors forbidden for $v_iz$ is at most $|N^2_{G}[v_iz] \setminus \{v_iz\}|$. Moreover,
by~\eqref{0830a}, if $d_G(u)\geq 5$, then  $M(v_iz)\leq |N^2_{G}[v_iz]\setminus \{v_iz\}|-1$. In any case,
$$M(v_iz)\leq d_{G}(v'_i)+\max\{4, d_{G}(u)-1\}+ (d_{G}(v_i)-3)\leq \Delta+\max\{4,\Delta-1\}+(\Delta-3)=3\Delta-4.$$
So we can always  find a free color for $v_iz$, for each $i\in[2]$.
This yields a $(3\Delta-3)$-coloring of $G$, a contradiction.
\end{proof}
  
Now we will complete the proof of the theorem using discharging: For each $v\in V(G^*)$, we 
let the {\em initial charge} $\ch(v):= d_{G^*}(v)$, and then will move
charge among vertices so that the {\em final charge} $\ch^*(v)$ is at least $8/3$ for  each $v\in V(G^*)$,
but the total sum of charge will be preserved during the entire process.
This will imply that 
\begin{equation}\label{0831}
 \sum_{v\in V(G^*)}d_{G^*}(v)=\sum_{v\in V(G^*)}ch(v)=\sum_{v\in V(G^*)}ch^*(v)\geq \frac{8}{3}|V(G^*)|,
\end{equation}
contradicting the fact that $\Mad(G^*)<\frac{8}{3}$.

The rules of discharging are the following.
\begin{itemize}
\item[(R1)] Each $(\Delta-1)^+_{G^*}$-vertex sends charge $2/3$ to each neighbor.
\item[(R2)] Each vertex $v$ with $4\leq d_{G^*}(v)\leq \Delta-2$  sends charge $1/3$ to every  $2_{G^*}$-neighbor.
\item[(R3)] Each $3_{G^*}$-vertex sends charge $1/3$ to every $2_{G^*}$-neighbor that is not special.
\end{itemize}

By~\eqref{0831}, the theorem will follow from the following claim:

\begin{claim}\label{cl: 8/3 after charge}
For every $v\in V(G^*)$, $ch^*(v)\geq \frac{8}{3}$.
\end{claim}
\begin{proof} We consider several cases depending on the degree of $v$.

 If $v$ is a special $2_{G^*}$-vertex, then it receives charge $2/3$ from its $(\Delta-1)^+_{G^*}$-neighbor by Rule (R1)
and gives out nothing. 
Thus $\ch^*(v) \geq \ch(v)+ 2/3 = 8/3$.

If $v$ is a $2_{G^*}$-vertex but not  special, then by Claim~\ref{cl: special 1}  it has  two $3^+_{G^*}$-neighbors, and
 each of them sends $v$ charge $1/3$ either by (R2) or by (R3). 
Thus $\ch^*(v)\geq \ch(v)+ 1/3+ 1/3 = 8/3$.

 If $d_{G^*}(v)=3$ and $v$ is adjacent to exactly $t$ vertices of degree two,
 then by Claim~\ref{cl: special 2}, at least $t-1$ of them are special. 
Thus by Rule (R3), $v$ sends out charge $1/3$ to at most one of its neighbors. Hence $\ch^*(v)\geq \ch(v)- 1/3 = 8/3$.

 If $4\leq d_{G^*}(v)\leq \Delta-2$, then by (R2), $v$ sends charge at most $1/3$ to each of its neighbors. 
Thus $\ch^*(v)\geq \ch(v) - \frac{d_{G^*}(v)}{3} \geq 2\ch(v)/3 \geq 8/3$.

Finally, if $d_{G^*}(v)\geq \Delta-1$, then by (R1), $v$ sends charge  $2/3$ to each of its neighbors. 
Thus $\ch^*(v) \geq \ch(v) - \frac{2d_{G^*}(v)}{3} =\ch(v)/3 \geq (\Delta-1)/3 \geq 8/3$ since $\Delta\geq 9$.
\end{proof}

\section{Graphs with maximum average degree less than three.}

In this section, instead of  Theorem~\ref{mad3} we  prove the following stronger result.

\begin{theorem}\label{lem: 3Delta}
Let $\Delta\geq 7$ be an integer
and let $G$ be a graph with no $3$-regular subgraph. 
If $\Delta(G)\leq \Delta$ and $\Mad(G)\leq 3$, then $\chi'_s(G)\leq 3\Delta$.
\end{theorem}

\subsection{Set-up of the proof and some notation}

To prove Theorem~\ref{lem: 3Delta}, we consider a counterexample $G$ with the
fewest $2^+$-vertices, and subject to this with the fewest edges.
Let $G^*$ be the graph  obtained from $G$ by deleting all vertices of degree $1$. 
Similarly to the proof of Theorem~\ref{8:3}, we will show that vertices of ``low" degree in $G^*$ have neighbors with ``high" degree, and based on this use discharging to prove that the average degree of $G^*$ is greater than $3$.

A feature not used in the previous proofs 
is the notion of potentials.
 For a graph $G$ and  $A\subseteq V(G)$, the {\em potential} of $A$, denoted $\rho_G(A)$, is defined as
$$\rho_G(A):= 3|A| - 2|E(G[A])|.$$
By definition, $\Mad(G)\leq  3$ if and only if $\rho_G(A) \geq 0$ for all $A\subseteq V(G)$.

The following fact on potentials is easy to check.
\begin{lemma}\label{lem: super-modular}
For a graph $G$ and any disjoint  $A,B\subset V(G)$,
$\rho_{G}(A) + \rho_{G}(B) = \rho_{G}(A\cup B) + \rho_{G}(A\cap B)+2|E_G(A\setminus B,B\setminus A)|$.
\end{lemma}

 A $3_{G^*}$-vertex is {\em poor} if it has exactly one $2_{G^*}$-neighbor.
For a poor $3_{G^*}$-vertex $u$, an edge $uw\in E(G^*)$ is the {\em $u$-sink}  if 
$d_{G^*}(w)=2$ and $N_{G^*}(u)- w$ contains a vertex $u'$ with $d_{G}(u')<\Delta$. 
An edge is a {\em sink} if it is a $u$-sink for some poor $3_{G^*}$-vertex $u$.
By definition, a poor $3_{G^*}$-vertex that is adjacent to two $\Delta_{G^*}$-vertices is not incident to a  sink.

Let $u$ be a $2_{G^*}$-vertex with $N_{G^*}(u) = \{v,w\}$. We say that $u$ is very poor, $v$ is a {\em sponsor of $u$}, and $w$ is a {\em rival of $u$}, if one of the following holds: 
\begin{itemize}

\item[(T1)]   $w$ is a $2_{G^*}$-vertex, or

\item[(T2)]  $w$ is a $3_{G^*}$-vertex with two $2_{G^*}$-neighbors in $G^*$ (including $u$), or

\item[(T3)] $w$ is a $4_{G^*}$-vertex and each vertex in $N_{G^*}(w)$ except at most one is either a $2_{G^*}$-vertex or a poor $3_{G^*}$-vertex.
\end{itemize}

 If a $2_{G^*}$-vertex is not very poor, then we say that it is {\em poor}.
See Figure~\ref{fig:poor} for an illustration. 
Poor vertices will be the recipients of charge in the discharging procedure.

\begin{figure}[h]
	\begin{center}
  \includegraphics[scale=0.8]{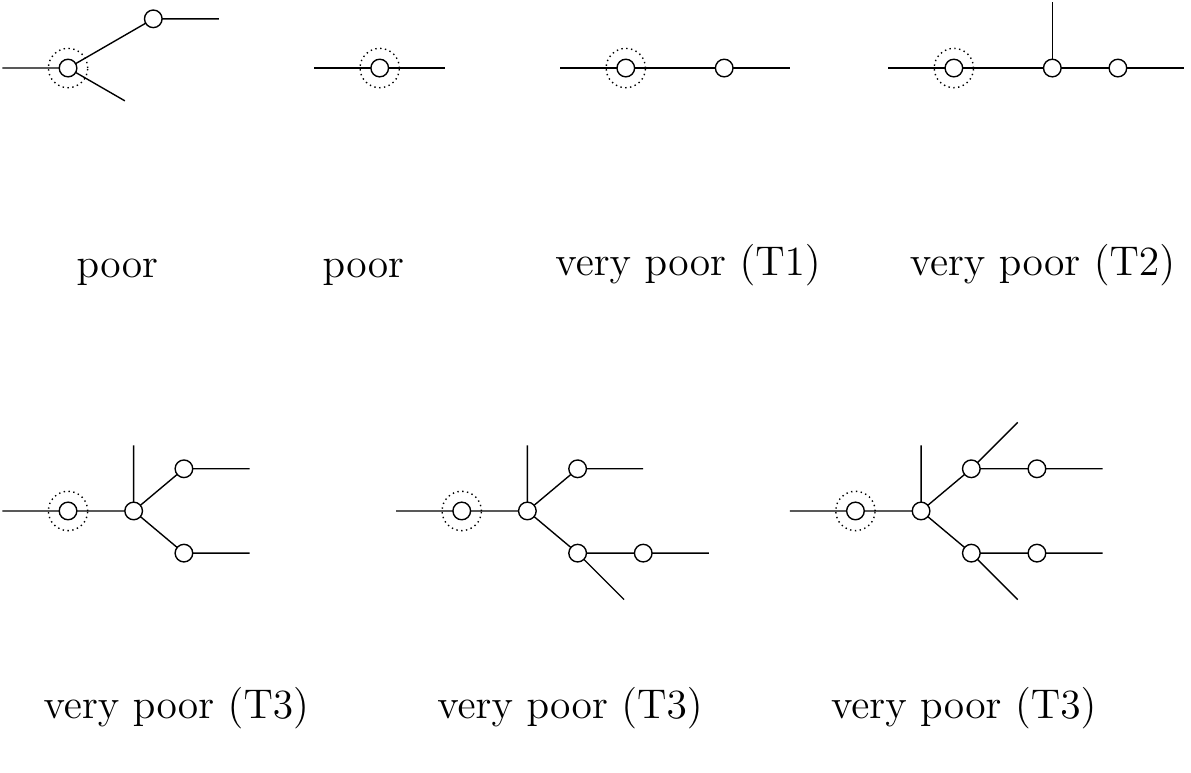}
  \caption{Poor vertices and very poor vertices.}
  \label{fig:poor}
	\end{center}
\end{figure}

If $u$ is a very poor vertex with a sponsor $v$ and a rival $w$, then 
the edge $uw$ is a {\em lower link} of $u$. 
Furthermore, if $w$ is the rival of type (T3) of $u$ and
$w'$ is a poor $3_{G^*}$-vertex in $N_{G^*}(w)$,
then $ww'$ a {\em semi-link} of $w$.  

Let $B(G)$ denote the set of all sinks in $G$. 
Similarly, let $S(G)$ and $S'(G)$ denote the set of all lower links in $G$ and the set of all semi-links in $G$, respectively.
By definition, all poor vertices and all very poor vertices are light. Also the rival of each very poor vertex is light.

\subsection{Structure of $G$ and $G^*$}
The next claim is in the spirit of Claims~\ref{cl: special 1} and~\ref{cl: special 2}.

\begin{claim}\label{cl: basic reducible}
Let $v\in V(G^*)$.
\begin{enumerate}[label=(\roman*)]
\item  If  
\begin{equation}\label{s1}
 \sum_{u\in N_{G^*}(v)} d_{G}(u)\leq 2\Delta + d_{G^*}(v),
 \end{equation}
  then $d_{G}(v)=d_{G^*}(v)$.
\item  $d_{G^*}(v)\geq 2$.
\item If $d_{G^*}(v)=2$, then $d_{G}(v)=2$.
\item If $d_{G^*}(v)=3$ and $v$ has a neighbor $u\in V(G^*)$ with $d_{G}(u)\leq 3$, then $d_{G}(v)=3$. 
In particular, if $v$ is poor $3_{G^*}$-vertex, then  $d_{G}(v)=3$. 
\end{enumerate}
By (ii), $\delta(G^*)\geq 2$.
\end{claim}
\begin{proof}
(i) Suppose~\eqref{s1} holds. If $d_{G}(v) \neq d_{G^*}(v)$, then $v$ has a $1_G$-neighbor $w$. 
By the minimality of $G$, the graph $G-w$ has a $3\Delta$-coloring $f$, which is a partial $3\Delta$-coloring of $G$. However by~\eqref{s1}, 
$$|N^2_{G}[vw]| \leq d_{G}(v) + \sum_{u\in N_{G}(v)} (d_{G}(u)-1) = d_{G}(v) + \sum_{u\in N_{G^*}(v)} d_{G}(u) - d_{G^*}(v) \leq d_{G}(v)
 + 2\Delta \leq 3\Delta.$$
 Thus we can extend $f$ to $vw$,  a contradiction to the definition of $G$. This proves (i).
\newline

(ii) Assume that $v$ is a $1_{G^*}$-vertex. Let $u$ be the unique neighbor of $v$ in $G^*$. 
Since $v\in V(G^*)$, we have $d_G(v)\geq 2 > d_{G^*}(v)$. However, 
$\sum_{w\in N_{G^*}(v)} d_{G}(w) = d_{G}(u)  \leq \Delta $, and~\eqref{s1} holds. 
Together with  $d_G(u)> d_{G^*}(u)$, this contradicts (i). 
This proves (ii), which implies $d_{G^*}(v)\geq 2$.
\newline

(iii) If $d_{G^*}(v)=2$, then $\sum_{u\in N_{G^*}(v)} d_{G}(u) \leq 2\Delta$. Thus (i) implies (iii).
\newline
 
(iv) If $d_{G^*}(v)=3$ and $v$ has a neighbor $u\in V(G^*)$ with $d_{G}(u)\leq 3$, then
 $$\sum_{w\in N_{G^*}(v)} d_{G}(w) \leq (d_{G^*}(v)-1)\Delta+d_{G}(u) \leq 2\Delta + 3.$$
 Thus~\eqref{s1} holds, and (i) implies that $d_{G}(v)=d_{G^*}(v)=3$. The ``In particular" part follows from (ii)
 and the definition of a poor $3_{G^*}$-vertex.
\end{proof}

The next claim collects more properties of neighbors of poor and very poor vertices.

\begin{claim}\label{cl: reducible 1} Graph $G$ possesses the following properties.
\begin{enumerate}[label=(\roman*)]
\item If $u$ is either a poor or a very poor vertex and $uv$ is a lower link, a semi-link, or  a sink, then $d_{G}(v)=d_{G^*}(v)$ and $d_{G}(u) = d_{G^*}(u)$.
\item 
If $e$ is a sink, then $|N_{G}^2[e]|\leq 3\Delta$.

\noindent If $e$ is a lower link, then $|N_{G}^2[e]\setminus B(G)|\leq 3\Delta-1$. 

\noindent If $e$ is a semi-link, then $|N_{G}^2[e]\setminus (B(G)\cup S(G))|\leq 3\Delta-1$.

\item If $v$ is a sponsor of a very poor vertex $u$, then $v$ is a $\Delta_{G^*}$-vertex.
\item If $u$ is a poor $3_{G^*}$-vertex, then it is adjacent to at least one $4^+_{G^*}$-vertex.

\item If $w$ is a rival of a very poor vertex $u$, then all but at most one neighbor of $w$ is pale.
\end{enumerate}
\end{claim}
\begin{proof}
(i) Claim~\ref{cl: basic reducible}~(iii) and (iv) implies that $d_{G}(u)=d_{G^*}(u)$.
 If $d_{G^*}(v)\leq 3$, then Claim~\ref{cl: basic reducible}~(iv) implies $d_{G}(v)=d_{G^*}(v)$. 
Otherwise, $u$ is either a very poor $2_{G^*}$-vertex of type (T3) or a poor $3_{G^*}$-vertex where $v$ is a rival of a very poor vertex of type (T3).
In any case, this implies $d_{G^*}(v)=4$. 
Assume $N_{G^*}(v)= \{u,u',u'',z\}$ where each of $u',u''$ is either a very poor $2_{G^*}$-vertex or a 
poor $3_{G^*}$-vertex. Again Claim~\ref{cl: basic reducible} (iii) and (iv) implies that $d_{G}(p)=d_{G^*}(p)$ for $p\in \{u,u',u''\}$.
Then $d_{G}(z)+ d_{G}(u)+ d_{G}(u')+ d_{G}(u'') \leq \Delta + 2+3+3 \leq 2\Delta + d_{G^*}(v)$, in other words,~\eqref{s1} holds. Thus Claim~\ref{cl: basic reducible}~(i) implies that $d_{G}(v)=d_{G^*}(v)$. \newline

(ii) Assume $e=uv$ is a $u$-sink.  This means  $u$ is a $3_{G^*}$-vertex, $v$ is a $2_{G^*}$-vertex, and
 at least one vertex in $N_{G^*}(u)\setminus \{v\}$ has degree less than $\Delta$ in $G$.
 Then  $d_{G}(u)=d_{G^*}(u)=3$ by Claim~\ref{cl: basic reducible}~(iv). Thus 
$$|N^2_G[e]| \leq \sum_{w \in N_{G^*}(u)\setminus \{v\}} d_{G}(w) + \sum_{w\in N_{G^*}(v)\setminus \{u\}} d_{G}(w) + |\{e\}| \leq 2\Delta-1 + \Delta +1 \leq 3\Delta.$$

Assume now $e=uv$ is a lower link of $u$.
Then by definition and (i), $d_G(u)=2$ and $2\leq d_{G}(v)\leq 4$.
If $d_{G}(v)=4$, let $t$ be the number of poor $3_{G^*}$-vertices in $N_{G}(v)$; then $0\leq t\leq 3$.
Depending on the type (T1)--(T3) of $v$ we have 
$$
|N^2_{G}[uv]|\leq 1+ \sum_{y\in (N_{G}(u)\cup N_{G}(v))\setminus\{u,v\}} d_{G}(y)\leq \left\{ \begin{array}{ll}
2\Delta + 1 & \text{ in case of (T1),} \\
2\Delta +3 & \text{ in case of (T2),} \\
2\Delta + 5 + t & \text{ in case of (T3) with } t\leq 2, \\
\Delta + 10 &  \text{ in case of (T3) with } t=3.
\end{array}\right.
$$
Since $d_{G^*}(v)\leq 4 <\Delta$, each poor $3_{G^*}$-vertex $w$ in $N_{G^*
}(v)$ is incident to  a sink. Thus 
$$
|N^2_{G}[uv]\setminus B(G)|\leq \left\{ \begin{array}{lll}
2\Delta + 1 &\leq 3\Delta-1& \text{ in case of (T1),} \\
2\Delta +3 &\leq 3\Delta -1& \text{ in case of (T2),} \\
2\Delta + 5 + t -t &\leq 3\Delta-1& \text{ in case of (T3).} \\
\end{array}\right.
$$

If $e=uv$ is a semi-link of $v$, then, by definition,  $v$ is the rival of some very poor vertex $w$ 
and $u$ is a poor $3_{G^*}$-vertex. So by (i), $d_{G}(v)=d_{G^*}(v)=4$.
Then each poor $3_{G^*}$-vertex $u'\in N_{G^*}(v)$ is incident to a sink, since $v\in N_{G^*}(u')$ satisfies $d_{G}(v)< \Delta$.
Let $t$ be the number of poor $3_{G^*}$-vertices in $N_{G^*}(v)$.
$$
|N^2_{G}[uv]|\leq 1+ \sum_{y\in (N_{G}(u)\cup N_{G}(v))\setminus\{u,v\}} d_{G}(y) \leq 2\Delta + 2(3-t)+ 3t.
$$
Since $N^2_{G}[uv]$ contains at least $t$ sinks and $3-t$ lower links,  $|N^2_{G}[uv]\cap (B(G)\cup S(G))|\geq 3$.
Hence  $|N^2_{G}[ww'] \setminus (B(G)\cup S(G))| \leq 2\Delta +6 +t-3 \leq 3\Delta-1$ since $\Delta\geq 7$ and $1\leq t\leq 3$.\newline

(iii) Suppose $v$ is a sponsor of a very poor vertex $u$ and
 $d_{G^*}(v)\leq \Delta-1$. Let $w$ be the rival of $u$.
Note that in each of the cases (T1)--(T3), $\Gamma_{G}(w)\setminus (S(G)\cup S'(G))$ contains at most one edge;   let this edge be $e'$.
Let $H$ be the graph obtained from $G-uw$ by adding, if necessary, leaves adjacent only to $v$ so that $d_H(v)=\Delta$. Then
$d_H(u)=1$ and since $d_{G^*}(v)\leq \Delta-1$, $v$ has a $1_H$-neighbor $x$ distinct from $u$.

Adding  leaves to a graph does not increase the maximum average degree, if it is at least $2$. It also does a not create new $3$-regular subgraph.
So $\Mad(H)\leq 3$ and $H$ has no $3$-regular subgraphs.
Since $H$ has fewer $2^+$-vertices than $G$, it has a $3\Delta$-coloring $f'$.
Since $x$ and $u$ are symmetric in $H$ and $f'(vx)\neq f'(vu)$, we may assume that $f'(vu)\neq f'(e')$ by changing colors of $vx$ and $vu$ if necessary. 

Let $f(e):= f'(e)$ for each edge $e \in E(H)\setminus (S(G)\cup S'(G)\cup B(G))$. Then $f\vert_{E(G)}$ is a partial $3\Delta$-coloring of $G$ 
since $f'(vu)\neq f'(e')$. 
By (ii) and the fact that $uw\in S'(G)\cup S(G)\cup B(G)$, $(S'(G),S(G),B(G))$ is an $(f\vert_{E(G)},3\Delta)$-degenerate sequence for $G$. Thus we conclude that $G$ is $(f\vert_{E(G)},3\Delta)$-degenerate. Thus Lemma~\ref{lem: degenerate} implies that 
 $\chi'_s(G)\leq 3\Delta$,  a contradiction. This proves (iii). \newline

(iv) Suppose that the neighbors of a poor $3_{G^*}$-vertex $u$ are $v_1,v_2$ and $v_3$, and $d_{G^*}(v_i)\leq 3$ for $i\in[3]$.
By the definition of a poor $3_{G^*}$-vertex, we may assume that 
 $d_{G^*}(v_1)=2$ and $d_{G^*}(v_2)=d_{G^*}(v_3)=3$. 
By Claim~\ref{cl: basic reducible}, $d_{G}(w)=d_{G^*}(w)$ for $w\in \{u,v_1,v_2,v_3\}$. 
Consider $H:= G - u$, which has fewer vertices of degree at least $2$ than $G$.
The minimality of $G$ implies that $H$ has a $3\Delta$-coloring $f$.
By the construction of $H$, $f$ is a partial $3\Delta$-coloring of $G$.
Note that $|N^2_{G}[uv_i]|\leq 2\Delta +6 \leq 3\Delta$ for $i\in[3]$. 
Thus $(uv_1,uv_2,uv_3)$ is an $(f,3\Delta)$-degenerate sequence for $G$, and so $G$ is $(f,3\Delta)$-degenerate. 
So Lemma~\ref{lem: degenerate} implies that $\chi'_s(H)\leq 3\Delta$, a contradiction.
 This proves (iv). \newline

(v) By definition, a poor $3_{G^*}$-vertex and a $2_{G^*}$-vertex are pale. 
Since each neighbor of $w$ possibly except one is either a poor $3_{G^*}$-vertex or a $2_{G^*}$-vertex, (v) follows.
\end{proof}

\begin{claim}\label{cl: no two very poor}
No vertex is a sponsor of two distinct very poor vertices.
\end{claim}
\begin{proof}
Assume that a vertex $v$ is a sponsor of  distinct very poor vertices $u$ and $u'$. 
By Claim~\ref{cl: basic reducible}(iii), both $u$ and $u'$ are $2_{G}$-vertices. 
Let $N_{G^*}(u)=\{v,w\}$ and  $N_{G^*}(u')=\{v,w'\}$.
By Claim~\ref{cl: reducible 1}(iii), $v$ is a $\Delta_{G^*}$-vertex. 
 \newline

\noindent {\bf\em Case 1.} $vw\in E(G)$ (this includes the case $w=u'$).
Consider the graph $H:= G- uw$. Since $H$ is a proper subgraph of $G$, 
it satisfies the conditions of Theorem~\ref{lem: 3Delta} and  contains  fewer $2^+$-vertices than $G$. 
So by the minimality of $G$, the graph $H$ has a $3\Delta$-coloring $f$. Since $vw\in E(G)$,
 $f$ is a partial $3\Delta$-coloring of $G$, where only $uw$ is not colored. By the definition
 of very poor vertices, $d_{G^*}(w)\leq 4$ and $v$ is the only possible neighbor of $w$ that is neither poor nor very poor.
Hence by Claims~\ref{cl: basic reducible} and~\ref{cl: reducible 1}(i),  
 $$ |N^2_{G}[uw]|\leq 1+d_G(v)+\sum_{x\in N_G(w)-\{v,u\}}d_G(x)\leq 1+\Delta+2(3)<3\Delta.
 $$
Thus we can extend $f$ to $uw$,  a contradiction. \newline

\noindent {\bf\em Case 2.} $vw\notin E(G)$ and $w=w'$.
Consider the graph $H$ obtained from $G-u$ by deleting the edge $u'w$  and adding the edge $vw$. 
Then $H$ has fewer $2^+$-vertices than $G$. Suppose $V(H)$ contains a set $A$ with $\rho_{H}(A)<0$. 
Since $\rho_G(A)\geq 0$, $wv \in E(H[A])$, so $w, v\in V(H)$.
Also we may assume that $u'\notin A$, since $\rho_{H}(A-u') \leq \rho_{H}(A)$. 
However, since each of $u$ and $u'$ is adjacent to each of $v$ and $w$, the graph $G[A\cup \{u,u'\}]$ has $4$ more edges than $G[A]$. So
\begin{align*}
\rho_{G}(A\cup \{u,u'\}) &= \rho_{G}(A)+2(3)-4(2)= (\rho_{H}(A)+2) +6-8=\rho_H(A)< 0,
\end{align*}
  a contradiction. Thus $\rho_{H}(B)\geq 0$ for any $B\subseteq V(H)$.
Similarly, if $H$ contains a $3$-regular subgraph $H'$, then $H'$  contains both $v$ and $w$.
This means $w$ has two neighbors in $G$ that are not pale. This contradicts
  Claim~\ref{cl: reducible 1}~(v).
Thus Lemma~\ref{lem: 3-reg light} implies that $H$ has no $3$-regular subgraphs  containing $w$.
Hence  $H$  has no $3$-regular subgraphs at all.
So $H$  satisfies the conditions of Theorem~\ref{lem: 3Delta} and by the minimality of $G$, $H$ has a $3\Delta$-coloring $f'$.
 Let 
$$f(e) := \left\{ \begin{array}{ll}
f'(e) & \text{if } e\in E(G)\cap E(H), \\
f'(vw) & \text{if } e= uv.
\end{array}\right.$$
Since $vw\in E(H)$, colors $f(uv)$ and $f(u'v)$ are disjoint from $f(\Gamma_{G}(w))$. So $f\vert_{E(G)}$ is a partial $3\Delta$-coloring of $G$ 
and 
the only non-colored edges  are $wu$ and  $wu'$.
Similarly to the end of the proof of Case 1, $d_{G^*}(w)\leq 4$ and at most one  neighbor of $w$,   is neither poor nor very poor.
Hence by Claims~\ref{cl: basic reducible} and~\ref{cl: reducible 1}(i),  for $y\in \{u,u'\}$
 $$ |N^2_{G}[yw]|\leq 2+d_G(v)+\sum_{x\in N_G(w)-\{u',u\}}d_G(x)\leq 2+\Delta+\Delta +3<3\Delta.
 $$
Thus we can extend $f$ to $uw$ and $u'w$,  a contradiction. \newline

\noindent {\bf\em Case 3.}  $|\{u,u',w, w'\}|=4$ and $vw,vw'\notin E(G)$.
Since every rival of a very poor vertex is incident to at most one edge that is not in $S(G)\cup S'(G)$, let $e'$ be the unique edge incident to $w'$ such that $e'\notin S(G)\cup S'(G)$, if it exists. 
Similarly to Case 2, consider the graph $H_1$ obtained from $G-u$ by deleting the edge $u'w'$  and adding the edge $vw$. 

If $H_1$ has a $3\Delta$-coloring $f_1$, then we may assume that $\{f_1(vw),f_1(vu')\}=\{\alpha,\beta\}$ with $\alpha\neq \beta$ 
and $f_1(e')\neq \beta$.
Let 
$$f(e) := \left\{ \begin{array}{ll} 
f_1(e) &\text{if } e\in E(G) \setminus (S(G)\cup S'(G) \cup B(G) \cup \{uv,vu'\} ) \\
\alpha &\text{if } e=uv \\
\beta &\text{if } e=vu'
\end{array} \right.$$
Then $e'$ is the only edge in $\Gamma_{H_1}(w')$ that is colored by $f$. Also  $\alpha,\beta \notin f(\Gamma_{H_1}(w))$, since
(due to the edge $vw$) every edge in $\Gamma_{H_1}(w)$ is distance at most one from $vu'$ and $vw$ in $H_1$. 
Thus $f$ is a partial $3\Delta$-coloring of $G$, and the uncolored edges are exactly the edges in $S(G)\cup S'(G)\cup B(G)$. (Note that $wu,w'u'\in S(G)$.) 
Hence
Claim~\ref{cl: reducible 1}~(ii) implies that $(S'(G),S(G),B(G))$ is an $(f,3\Delta)$-degenerate sequence for $G$; thus $G$ is $(f,3\Delta)$-degenerate, a contradiction.

Since $H_1$ has fewer $2^+$-vertices than $G$, this means that $H_1$ does not satisfy the conditions of our theorem. This means that
 either there exists a set $A\subseteq V(H_1)$ with $\rho_{H_1}(A)<0$ or  $H_1$ has  a $3$-regular subgraph $H'_1$. 
If the latter holds, then $H'_1$ must contain the edge $wv$ since $G$ has no $3$-regular subgraphs. Then the neighbors of $w$ in $H'_1-v$
are not pale in $H_1$ and hence in $G$, a contradiction to  Claim~\ref{cl: reducible 1}~(v). 
We conclude that there exists a set $A_1\subseteq V(H_1)$ with $\rho_{H_1}(A_1)<0$. 
Since $\rho_{G}(A_1)\geq 0$, we know that $w,v \in A_1$ and $u'\notin A_1$. Then 
$$\rho_{G}(A_1 \cup \{u\}) = \rho_{G}(A_1) +3 - 2\cdot 2 = \rho_{H_1}(A_1)+2 + 3 - 4 = \rho_{H_1}(A_1)+1\geq 0.$$
It follows that $\rho_{H_1}(A_1) = -1$,   $\rho_{G}(A_1\cup\{u\})=\rho_{H_1}(A_1)+1=0$ and 
$\rho_{G}(A_1) =\rho_{G}(A_1\cup\{u\}) -3+2(2)=1$. Also $w'\notin A_1$, since otherwise we have $\rho_{G}(A_1\cup \{u,u'\}) = \rho_{G}(A_1) + 2\cdot 3 - 4\cdot 2 = -1$, a contradiction.

By symmetric argument, we can also find a set $A'_1$ such that $w',v\in A'_1$, $u,w\notin A'_1$, $\rho_{G}(A'_1\cup\{u'\})=0$ and  $\rho_{G}(A'_1)=1$. 
Then by Lemma~\ref{lem: super-modular},
$$\rho_G((A_1\cup\{u\})\cap (A'_1\cup\{u'\}))+\rho_G((A_1\cup\{u\})\cup (A'_1\cup\{u'\}))\leq \rho_G(A_1\cup\{u\})+\rho_G (A'_1\cup\{u'\})=0+0=0.$$
Since $(A_1\cup\{u\})\cap (A'_1\cup\{u'\})=A_1\cap A'_1$, we conclude that
\begin{equation}\label{s2}
\rho_G(A_1\cap A'_1)=0.
\end{equation}
If $ww'\in E(G)$, then by Lemma~\ref{lem: super-modular}, 
\begin{align*}
\rho_G(A_1\cup A'_1\cup\{u,u'\})&= \rho_{G}(A_1\cup\{u\})+\rho_{G}(A'_1\cup\{u'\})-\rho_{G}((A_1\cup\{u\})\cap (A'_1\cup\{u'\}))\\
&\enspace \enspace -2|E_G(\{u\}\cup A_1\setminus A'_1,\{u'\}\cup A'_1\setminus A_1)| \\
&\leq 0+0-0-2(1)=-2,
\end{align*}
a contradiction. Thus $ww'\notin E(G)$.

Now we consider the graph $H_2$ obtained from $G - \{wu, w'u'\}$ by adding the edge $ww'$. 
Recall that $e'$ is the unique edge incident to $w'$ such that $e' \notin S(G)\cup S'(G)$, if it exists. 
Let $e''$ be the unique edge incident to $w$ such that $e''\notin S(G)\cup S'(G)$, if it exists. 

Assume  $H_2$ has  a $3\Delta$-coloring $f_2$.
Then $f_2(uv)\neq f_2(u'v)$. Since $e'$ and $e''$ are distance one from each other in $H_2$, $f_2(e')\neq f_2(e'')$. 
We may assume that $f_2(e')\neq f_2(u'v)$ and $f_2(e'')\neq f_2(uv)$ by switching the colors of $uv$ and $u'v$ if necessary. 

Then we let 
$$f(e) = f_2(e) \text{ for } e\in E(G)\setminus (S(G)\cup S'(G)\cup B(G)).$$
The only edge in $N^2_{G}(uv)$ incident to $w$ and colored in $f_2$ is $e''$, and the only edge in $N^2_{G}(u'v)$ incident to $w'$ colored in $f_2$ is $e'$. 
Since $f_2(e'')\neq f_2(uv)$ and $f_2(e')\neq f_2(u'v)$, $f$ is a partial $3\Delta$-coloring of $G$.
Now Claim~\ref{cl: reducible 1}~(ii) implies that $(S'(G),S(G),B(G))$ is an $(f,3\Delta)$-degenerate sequence for $G$. Thus $G$ is $(f,3\Delta)$-degenerate, a contradiction. So $H_2$ must not have a $3\Delta$-coloring. 

Since $H_2$ contains fewer $2^+$-vertices than $G$, this means that $H_2$ does not satisfy the conditions of our theorem. 
So either there exists a set $A_2\subseteq V(H_2)$ with $\rho_{H_2}(A_2)<0$ or  $H_2$
has a $3$-regular subgraph $H'_2$. In the latter case,  $H'_2$ must contain the edge $ww'$, since $G$ contains no $3$-regular subgraphs. 
Then the neighbors of $w$ in $H'_2-w'$
are not pale in $H_2$ and hence in $G$, a contradiction to  Claim~\ref{cl: reducible 1}~(v). 
We conclude that there exists a set $A_2\subseteq V(H_2)$ with $\rho_{H_2}(A_2)<0$. 
Since $\rho_{G}(A_2)\geq 0$, we know that $w,w' \in A_2$. We may also assume that $u,u'\notin A_2$ since $\rho_{H_2}(A_2\setminus\{u,u'\})\leq \rho_{H_2}(A_2)$. Then $\rho_{G}(A_2)  = \rho_{H_2}(A_2)+2  \leq 1$.  Also $v\notin A_2$ since otherwise $\rho_{G}(A_2\cup \{u,u'\}) = \rho_{G}(A_2)+ 2\cdot 3 - 4\cdot 2 \leq -1$, a contradiction.
Thus we have a set $A_2$ with $\rho_{H_2}(A_2)\leq 1$, $w,w'\in A_2$ and $u,v,u'\notin A_2$.

Since $\{v,w,w'\}\subseteq A_1\cup A_2$ and $\{u,u'\}\cap (A_1\cup A_2)=\emptyset$,
$\rho_G(A_1\cup A_2\cup \{u,u'\})=\rho_G(A_1\cup A_2)+2(3)-4(2)=\rho_G(A_1\cup A_2)-2$. It follows that $\rho_G(A_1\cup A_2)\geq 2$.
So by
 Lemma~\ref{lem: super-modular}, 
$$\rho_{G}(A_1\cap A_2) =\rho_{G}(A_1)+\rho_{G}(A_2) - \rho_{G}(A_1\cup A_2)\leq 1 + 1 - 2 = 0.$$
Then by~\eqref{s2} and again
 by Lemma~\ref{lem: super-modular} 
$$\rho_{G}((A_1\cap A_1') \cup (A_1\cap A_2)) \leq \rho_{G}(A_1\cap A_1') + \rho_{G}(A_1\cap A_2)  =  0 + 0 =0.$$
Yet, $(A_1\cap A_1') \cup (A_1\cap A_2)$ contains $v$ and $w$ and does not contain $u$. So
$$\rho_{G}( (A_1\cap A_1') \cup (A_1\cap A_2) \cup \{u\}) = \rho_{G}((A_1\cap A'_1) \cup (A_1\cap A_2)) + 3 - 2( 2) = -1,$$
 a contradiction to the choice of $G$. This proves the lemma.
\end{proof}

\begin{claim}\label{cl: reducible 2}
Let $v\in V(G^*)$.
\begin{enumerate}[label=(\roman*)]
\item If $d_{G^*}(v)< \Delta$, then at least one vertex in $N_{G^*}(v)$ is neither poor nor very poor.
\item If $d_{G^*}(v)=\Delta$, then either at least one vertex in $N_{G^*}(v)$ is neither poor nor very poor or $v$ is not a sponsor of a very poor vertex.
\end{enumerate}
\end{claim}
\begin{proof}
Let $d_{G^*}(v)=s$ and $N_{G^*}(v) =\{u_1,\dots, u_s\}$ where each $u_i$ is either poor or very poor. 
Further assume that $u_1,\dots, u_t$ ($t\leq s$) are the poor $3_{G^*}$-vertices in $N_{G^*}(v)$.
By  Claim~\ref{cl: basic reducible}, 
\begin{equation}\label{s3}
\mbox{\em $d_{G^*}(x)=d_G(x)$ for each $x\in N_{G^*}[v]$.}
\end{equation}
In particular, $d_{G^*}(v)=d_G(v)=s$.
For $i\in [t]$, since $u_i$ is a poor $3_{G^*}$-vertex, let $N_G(u_i)=\{v, u'_i,u''_i\}$, where $u''_i$ is the unique 
$2_{G^*}$-vertex in $N_{G}(u_i)$, $N_G(u''_i)=\{u_i,x_i\}$,  and let $e_i:= u_iu''_i$. 
Note that $e_i$ may or may not be a sink.
For $i\in[s]\setminus[t]$, we know $d(u_i)=2$, so let $N_{G}(u_i) = \{v,u'_i\}$. 
\newline

(i) Suppose $s < \Delta$. Recall that $d_{G}(v)=s < \Delta$.
For each $i\in[t]$, the   poor $3_{G^*}$-vertex $u_i$ is adjacent to the $(\Delta-1)^{-}_{G}$-vertex $v$ and hence  is incident to exactly one sink $e_i$.  

Let  $H := G-v$.  By the minimality of $G$, the graph $H$ has a $3\Delta$-coloring $f'$.
Let 
$$f(e) := f'(e) \text{ for } e\in E(G) \setminus B(G).$$
Then $f$ is a partial $3\Delta$-coloring of $G$.
Consider the ordering $(vu_1,\dots, vu_s, B(H))$ of the edges not colored by $f$.
First, since $d_G(v)\leq \Delta-1$, for each $i\in [t]$,
\begin{align*}
|N^2_{G}[vu_i]\setminus (B(G)\cup\{vu_{i+1},\dots, vu_s\})|&\leq d_{G}(u'_i)+ d_{G}(u''_i) + \sum_{j\neq i}|N_{G}(u_j)\setminus B(G)|+|\{vu_i\}|\\ &\leq \Delta + 2 + 2(d_{G}(v)-1) +1\leq 3\Delta-1.
\end{align*}
Similarly, if $t<i\leq s$, then
\begin{align*}
|N^2_{G}[vu_i]\setminus (B(G)\cup\{vu_{i+1},\dots, vu_s\})|&\leq  d_{G}(u'_i) + \sum_{j\neq i}|N_{G}(u_j)\setminus B(G)|+|\{vu_i\}|\\ 
&\leq  \Delta + 2(d_{G}(v)-1) +1\leq 3\Delta-3.
\end{align*}
Claim~\ref{cl: reducible 1}~(ii) implies that $|N^2_{G}[e]| \leq 3\Delta$ for each $e\in B(G)$.
Therefore, 
$(vu_1,\dots, vu_s, B(G))$ is an $(f,3\Delta)$-degenerate sequence for $G$.
Hence $G$ is $(f,3\Delta)$-degenerate,  a contradiction to Lemma~\ref{lem: degenerate} and the assumption that $G$ does not have a $3\Delta$-coloring. 
\newline

(ii) Assume  
$s=\Delta$ and $u_\Delta$ is very poor, so that $t<s$.
Let $u'$ be the rival of $u_\Delta$ and let $e'$ be the unique edge incident to $u'$ where $e'\notin S(G)\cup S'(G)$, if it exists. 

Consider $H := G-v$. By the minimality of $G$, graph $H$ has a $3\Delta$-coloring $f'$.
Let 
$$f(e) := f'(e) \text{ for } e\in E(G) \setminus (\{e_1,\dots, e_t\}\cup S(G)\cup S'(G)\cup B(G)).$$
Then $f$ is a partial $3\Delta$-coloring of $G$ since $N^2_{H}[e] = N^2_{G}[e]\setminus \Gamma_G(v)$ for each $e\in E(H)$.
Let $B'(G) := B(G)\setminus\{e_1,\dots, e_t\}$ and $P(G) := S'(G)\cup S(G)\cup B'(G).$ 

Consider the sequence $(vu_1,\dots, vu_{\Delta-1}, e_1,\dots, e_t, vu_\Delta, S'(G), S(G),B'(G))$ of edges not colored by $f$.
We want to show that it  is an $(f,3\Delta)$-degenerate sequence for $G$.
Note that no edge incident to $u_\Delta$ is  colored by $f$.
First, for $i\in[t]$,
\begin{align*}
|N^2_{G}[vu_i]\setminus (\{e_1,\dots, e_t,vu_\Delta\}\cup P(G)) |&\leq d_{G}(u'_i)+ (d_{G}(u''_i)-1) + 
\sum_{j\not\in \{i,\Delta\}}|N_{G}(u_j)\setminus\{e_j\}|+|\{vu_i\}|\\ &\leq \Delta + 1 + 2(\Delta-2) +1\leq 3\Delta-2.
\end{align*}
If $t<i\leq \Delta-1$, then
\begin{align*}
|N^2_{G}[vu_i]\setminus (\{e_1,\dots, e_t,vu_\Delta\}\cup P(G))|&\leq  d_{G}(u'_i) + \sum_{j\not\in \{i,\Delta\}}|N_{G}(u_j)
\setminus\{e_1,\dots, e_t\}|+|\{vu_i\}|\\ &\leq  \Delta + 2(\Delta-2) +1\leq 3\Delta-3.
\end{align*}
For $i\in [t]$,
\begin{align*}
|N^2_{G}[e_i]\setminus (\{vu_\Delta\}\cup P(G))|&\leq d_{G}(x_{i})+ d_{G}(u'_{i}) + |N_{G}(v)\setminus \{vu_\Delta\}| + |\{e_i\}| \\ 
&\leq \Delta +\Delta +  \Delta -1 + 1 \leq 3\Delta.
\end{align*}
Also 
\begin{align*}
|N^2_{G}[vu_\Delta]\setminus P(G)|&\leq \sum_{i=1}^{\Delta-1} d_{G}(u_i) + |\{e'\}| + |\{vu_\Delta\}| \leq 3(\Delta-1) + 2  \leq 3\Delta-1.
\end{align*}
These inequalities together with Claim~\ref{cl: reducible 1}~(iii) imply that the  sequence \\
$(vu_1,\dots, vu_{\Delta-1}, e_1,\dots, e_t, vu, S'(G),S(G),B'(G))$ is  an $(f,3\Delta)$-degenerate sequence for $G$. So, $G$ is $(f,3\Delta)$-degenerate,  a contradiction to Lemma~\ref{lem: degenerate}
 and the choice of $G$. This proves the claim.
\end{proof}

\begin{claim}\label{cl: 4-vtx}
Suppose $v$ is a poor $3_{G^*}$-vertex and $N_{G^*}(v)=\{v_1,v_2,v_3\}$, where  $d_{G^*}(v_1)=4, d_{G^*}(v_2)=3$,
and  $d_{G^*}(v_3)=2$. Then $v_1$ has at least two neighbors in $G^*$ where each of them is  neither poor nor very poor. 
\end{claim}
\begin{proof}
Suppose that under the conditions of the lemma, $N_{G^*}(v_1)=\{v,u_1,u_2,x\}$, where each of $u_1$ and $u_2$  is either poor or very poor. 
Let $N_{G^*}(v_2) = \{ v,y_1,y_2\}$ and $N_{G^*}(v_3)= \{v,z\}$.
Since $v,u_1$, and $u_2$ are all poor or very poor, their degrees in $G$ are the same as in $G^*$. 
Since $\Delta\geq 7$, 
 $$ 2\Delta+ d_{G^*}(v_1)=2\Delta + 4  > 10 + \Delta> 3+3+3+\Delta \geq d_{G}(v) + d_{G}(u_1) + d_{G}(u_2) + d_{G}(x).$$
 Thus Claim~\ref{cl: basic reducible}~(i) implies that $d_{G}(v_1)= d_{G^*}(v_1)=4$.
Also by Claim~\ref{cl: basic reducible}~(iv), $d_{G}(v_2)= d_{G^*}(v_2)=3$.

Consider $H:= G-v$. By the minimality of $G$, the graph $H$ has a $3\Delta$-coloring $f$.
By the construction of $H$, $f$ is a partial $3\Delta$-coloring of $G$.

Since $\Delta \geq 7$, 
\begin{align*}
|N^2_{G}[vv_1]\setminus\{vv_2,vv_3\}|&\leq d_{G}(x) + (d_{G}(v_2)-1)+ (d_{G}(v_3)-1)+ \sum_{i=1}^{2} d_{G}(u_i) +|\{vv_1\}| \\
&\leq \Delta + 2 + 1 + 2\cdot 3 +1 \leq 3\Delta,\\
|N^2_{G}[vv_2]\setminus\{vv_3\} |&\leq d_{G}(y_1) + d_{G}(y_2) + d_{G}(v_1)+ d_{G}(v_3)-1 + |\{vv_2\}| \\ &\leq 2\Delta + 4+1 + 1 \leq 3\Delta,\\
|N^2_{G}[vv_3]|&\leq d_{G}(z) + d_{G}(v_2)+ d_{G}(v_1) +|\{vv_3\}|\leq \Delta + 3+4+1 \leq 3\Delta.
\end{align*}
Thus $(vv_1,vv_2,vv_3)$ is an $(f,3\Delta)$-degenerate sequence for $G$. So $G$ is $(f,3\Delta)$-degenerate,  a contradiction to the 
choice of $G$.
\end{proof}

\subsection{Discharging}

 Since $G^*$ is a subgraph of $G$, we know $\Mad(G^*)\leq 3$ and $G^*$ does not contain $3$-regular subgraphs.
For every $v\in V(G^*)$, define {\em the initial charge} $\ch(v):= d_{G^*}(v)$. Since $\Mad(G^*)\leq 3$, we have
$
\sum_{v\in V(G^*)}\ch(v) \leq 3|V(G^*)|.
$
 We will move the charge among vertices without changing the total sum of charge according  to the discharging rules below.

\begin{itemize}
\item[(R1)] If a $\Delta_{G^*}$-vertex $v$ is a sponsor of a very poor vertex $u$, then $v$ gives charge $1$  to $u$. 

\item[(R2)] If $5\leq d_{G^*}(v)\leq \Delta$ and $u$ is a poor vertex in $N_{G^*}(v)$, then $v$ gives charge $1/2$  to $u$.

\item[(R3)] If a $4_{G^*}$-vertex $v$ is adjacent to a poor $2_{G^*}$-vertex $u$, then $v$ gives charge $1/2$  to $u$.

\item[(R4)] If a $4_{G^*}$-vertex $v$ is adjacent to a poor $3_{G^*}$-vertex $u$, then we do one of the following:
\begin{itemize}
\item[(R4A)] If $N_{G^*}(v)$ contain at least two vertices that are neither poor nor very poor, then $v$ gives charge $1/2$  to $u$.
\item[(R4B)] Otherwise, $v$ gives charge $1/4$  to $u$.
\end{itemize}

\item[(R5)] If a poor $3_{G^*}$-vertex $v$ is adjacent to a poor $2_{G^*}$-vertex $u$, then $v$ gives charge $1/2$  to $u$.
\end{itemize}
For every $v\in V(G^*)$, let $\ch^*(v)$ be the {\em final charge} of $v$, which is the charge of $v$ after the distribution. Since the total sum of charge did not change,
\begin{equation}\label{s5}
\sum_{v\in V(G^*)}\ch^*(v)= \sum_{v\in V(G^*)}\ch(v) \leq 3|V(G^*)|.
\end{equation}

See Figure~\ref{fig:rules} for an illustration of the discharging rules. 

\begin{figure}[h]
	\begin{center}
  \includegraphics[scale=0.725]{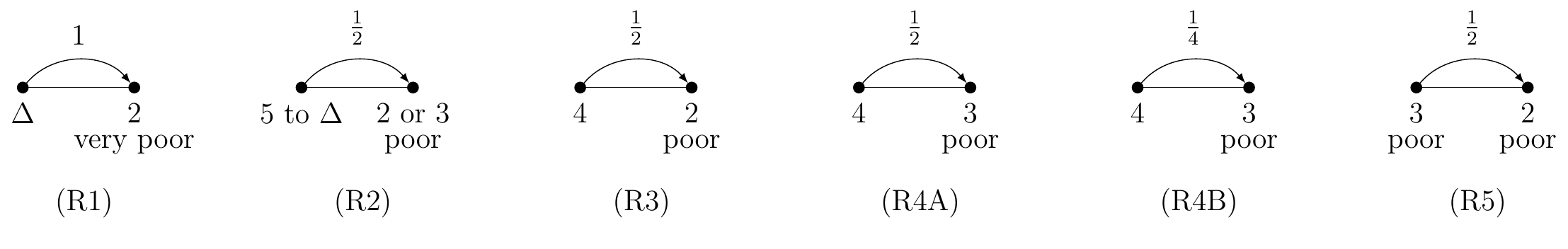}
  \caption{The discharging rules. }
  \label{fig:rules}
	\end{center}
\end{figure}

The next claim shows important properties of the final charge $\ch^*$.

\begin{claim}\label{cl: after charge}
Let $v\in V(G^*)$.
\begin{enumerate}[label=(\roman*)]
\item If $v$ is a $2_{G^*}$-vertex, then $\ch^*(v) = 3$.
\item If $v$ is a $3_{G^*}$-vertex, then $\ch^*(v) \geq 3$.
\item If $v$ is a $4_{G^*}$-vertex, then $\ch^*(v)\geq 3$. Moreover, if $\ch^*(v)=3$, 
then $v$ is adjacent to exactly two poor vertices and no very poor vertices. 
\item If $5\leq d_{G^*}(v)\leq \Delta-1$, then $\ch^*(v) \geq 3$. Moreover, if $\ch^*(v)=3$, 
then $d_{G^*}(v)=5$ and $v$  has exactly four poor 
neighbors.
\item If $v$ is a $\Delta_{G^*}$-vertex, then $\ch^*(v)> 3$. 
\end{enumerate}
\end{claim}
\begin{proof}
(i) Suppose $d_{G^*}(v)=2$. Then $v$ is either poor or very poor.
If $v$ is  very poor, then it is adjacent to its sponsor $w$, and
by Claim~\ref{cl: reducible 1}~(iii), $d_{G^*}(w)=\Delta$. Hence by (R1), $w$ gives charge $1$ to $v$, so $\ch^*(v)=3$. 

If $v$ is poor but not very poor, then both neighbors of $v$ are $3^+_{G^*}$-vertices, and each $3_{G^*}$-neighbor of $v$ is poor.
Thus, by rules (R2), (R3), or (R5), $v$ receives $1/2$ from each of its two neighbors in $G^*$.
 So  $\ch^*(v)=3$. 
\newline

(ii) Suppose $d_{G^*}(v)=3$. If $v$ is not  poor, then $v$ does not send or receive any charge, so $\ch^*(v)=\ch(v)=3$.
Now assume $v$ is poor.  
Then by Claim~\ref{cl: reducible 1} $(iv)$, $v$ has a $4^+_{G^*}$-neighbor.  By (R5), $v$ gives $1/2$ to its unique
$2_{G^*}$-neighbor, and by (R2) or (R4) receives at least $1/4$ from each $4^+_{G^*}$-neighbor. Thus, if
$\ch^*(v)<3$, then $v$ has only one $4^+_{G^*}$-neighbor, say
 $u$. Furthermore, by (R2), $d_{G^*}(u)=4$, and by (R4), $u$ has at most one neighbor in $G^*$ that is neither poor nor very poor.
 But this contradicts  
 Claim~\ref{cl: 4-vtx}. 
\newline
 
(iii) Suppose $d_{G^*}(v)=4$.
If $v$ is the rival of a very poor vertex $u$, then by the definition (T3), every $2_{G^*}$-neighbor of $v$ is very poor and hence $v$ does not give to its $2_{G^*}$-neighbors anything. 
In this case, it either gives $1/2$ to its unique $3_{G^*}$-neighbor by (R4A) 
or gives $1/4$ to each of its at most three poor neighbors by (R4B). 
In both cases, $\ch^*(v) \geq 4 - 3\cdot\frac{1}{4} >3$.

Otherwise, $v$ is not adjacent to any very poor vertices. 
Then by Claim~\ref{cl: reducible 2}~(i), $v$ is adjacent to at most three poor vertices. 
If $v$ has  three poor neighbors, then all of them are $3_{G^*}$-vertices, because in this case each $2_{G^*}$-neighbor of $v$ is very poor. 
Thus $v$ sends charge $1/4$ to its three poor neighbors by (R4B), so $\ch^*(v) \geq 4 - 3\cdot \frac{1}{4} >3$. The last possibility is that
$v$ has at most two poor neighbors. Since $v$ gives to any poor neighbor at most $1/2$, the only possibility to have
$\ch^*(v) \leq 3$ is that $v$ has exactly two poor neighbors and gives $1/2$ to each of them, in which case $\ch^*(v) = 3$.
This proves (iii).
\newline

(iv) Suppose $5\leq d_{G^*}(v)\leq \Delta-1$. 
By (R2), $v$ gives charge $1/2$ to each of its poor neighbors. 
By Claim~\ref{cl: reducible 2}~(i), the number of such neighbors is at most $d_{G^*}(v)-1$.
So
$$\ch^*(v)\geq \ch(v) - \frac{1}{2}(d_{G^*}(v)-1) = \frac{d_{G^*}(v)+1}{2}\geq 3.$$
Moreover, if $\ch^*(v)=3$, then $d_{G^*}(v)=5$ and $v$ is adjacent to exactly four poor neighbors.
\newline

(v) Suppose $d_{G^*}(v)=\Delta$. If $v$ has no very poor neighbors, then by (R2) it gives at most $1/2$
to each of its neighbors, and $\ch^*(v) \geq \ch(v)-\frac{d_{G^*}(v)}{2}=\frac{d_{G^*}(v)}{2}\geq \frac{\Delta}{2}>3$. Otherwise,
by Claim~\ref{cl: no two very poor}, it has only one very poor neighbor (to which it gives $1$ by (R1)),
but then by Claim~\ref{cl: reducible 2}~(ii), it has a neighbor that is neither poor nor very poor.
Thus by (R1) and (R2), again
$$\ch^*(v) \geq d_{G^*}(v)-1-\frac{\Delta-2}{2} = \Delta/2 > 3.$$
\end{proof}

Claim~\ref{cl: after charge} implies that $\sum_{v\in V(G^*)}\ch^*(v) \geq 3|V(G^*)|$ and together with~\eqref{s5}
yields
\begin{equation}\label{s6}
 \mbox{\em $\ch^*(v)=3$ for each $v\in V(G^*)$.}
\end{equation}

This yields the following facts.

\begin{claim}\label{cl: after structure}
Graph $G$ has the following properties.
\begin{enumerate}[label=(\roman*)]
\item $\Delta(G^*)\leq 5$.

\item $G^*$ has no very poor vertices.

\item $B(G)=\emptyset$.

\item If $v$ is a $5_{G^*}$-vertex, then $d_G(v)=d_{G^*}(v)$.

\item There are no $5_{G^*}$-vertices; in other words, $\Delta(G^*)\leq 4$.

\item There are no poor $3_{G^*}$-vertices.
\end{enumerate}
\end{claim}
\begin{proof}
(i) Claim~(i) follows from~\eqref{s6} and Claim~\ref{cl: after charge} (parts (iv) and (v)).
\newline

(ii) If $v$ is a very poor vertex in $G^*$, then  by Claim~\ref{cl: reducible 1}~(iii), it has a $\Delta_{G^*}$-neighbor, contradicting (i).
\newline

(iii) Suppose that $G^*$ has a $u$-sink $uw$, which means that $d_{G^*}(w)=2$, $d_{G^*}(u)=3$, and 
\begin{equation}\label{s7}
 d_G(u')+d_G(u'')\leq 2\Delta-1.
\end{equation}
where $N_{G^*}(u)=\{w,u',u''\}$ and $N_{G^*}(w)=\{u,w'\}$.
Recall that by Claim~\ref{cl: basic reducible}, $d_G(w)=d_{G^*}(w)=2$ and $d_G(u)=d_{G^*}(u)=3$.
Let $H$ be obtained from $G-uw$ by adding leaves adjacent only to $w'$ so that the degree of $w'$ in $H$ is $\Delta$.
Since $d_H(w)=1$, $H$  has fewer $2^+$-vertices than $G$. Also, $\Mad(H)\leq 3$ and $H$ has no $3$-regular subgraphs.
So by the minimality of $G$, the graph $H$ has a $3\Delta$-coloring $f$. By (i) and the fact that $d_H(w)=1$, vertex $w'$ has
at least $\Delta-4\geq 3$ $1_H$-neighbors, including $w$. Hence we can switch the colors of the pendant edges incident to $w'$ so
that $f(w'w)\notin \{f(uu'),f(uu'')\}$. Then $f\vert_{E(G)}$ is a partial  $3\Delta$-coloring of $G$, where the only non-colored edge
is $uw$. But by~\eqref{s7},
$$|N^2_G[uw]|\leq d_G(u')+d_G(u'')+d_G(w')+1\leq (2\Delta-1)+\Delta+1=3\Delta,$$
so we can extend $f$ to $uw$, a contradiction to the choice of $G$. This proves (iii).
\newline

(iv) Suppose  $d_{G^*}(v)=5$.  By Claim~\ref{cl: after charge}~(iv), we may assume that $N_{G^*}(v)=\{u_1,\ldots,u_5\}$, where
each of $u_1,u_2, u_3,u_4$ is either poor or very poor. In particular, by Claim~\ref{cl: basic reducible}, $d_G(u_i)=d_{G^*}(u_i)\leq 3$
for $i\in[4]$. Thus,
$$\sum_{i=1}^5d_G(u_i)\leq 3(4)+\Delta\leq 5+2\Delta=d_{G^*}(v)+2\Delta,$$
and Claim~\ref{cl: basic reducible}~(i) implies $d_G(v)=d_{G^*}(v)$.
\newline

(v)  Suppose  $d_{G^*}(v)=5$ and $N_{G^*}(v)=\{u_1,\ldots,u_5\}$, where
each of $u_1,u_2, u_3,u_4$ is either poor or very poor. If some $u_i$ is a poor $3_{G^*}$-vertex adjacent to a
$2_{G^*}$-vertex $w_i$, then by (iv), $u_iw_i$ is the $u_i$-sink. But this contradicts (iii). Thus 
\begin{equation}\label{s8}
 \mbox{\em
for each $i\in[4]$,
$u_i$ is a poor $2_{G^*}$-vertex. }
\end{equation}
So we may assume that $N_G(u_i)=\{v,w_i\}$ for $i\in [4]$.
Let $H$ be obtained from $G-vu_1$ by adding leaves adjacent only to $w_1$ so that the degree of $w_1$ in $H$ is $\Delta$.
Since $d_H(u_1)=1$, $H$  has fewer $2^+$-vertices than $G$. Also, $\Mad(H)\leq 3$ and $H$ has no $3$-regular subgraphs.
So by the minimality of $G$, the graph $H$ has a $3\Delta$-coloring $f$. By (i) and the fact that $d_H(u_1)=1$, the number of $1_H$-neighbors of the vertex $w_1$, including $u_1$, is at least $\Delta-4\geq 3$.
Hence we can switch the colors of the pendant edges incident to $w_1$ so that $f(w_1u_1)\neq f(vu_5)$. 

If $f(w_1u_1)\notin \{f(vu_2), f(vu_3),f(vu_4)\}$, then $f\vert_{E(G)}$ is a partial $3\Delta$-coloring of $G$, where the only uncolored edge
is $u_1v$. But in this case by~\eqref{s8},
\begin{equation}\label{s9}
|N^2_G[u_1v]|\leq d_G(w_1)+\sum_{i=2}^5d_G(u_i)+1\leq \Delta+3(2)+\Delta+1=2\Delta+7\leq 3\Delta,
\end{equation}
so we can extend $f$ to $u_1v$, a contradiction to the choice of $G$. 

Thus by symmetry, we may assume  $f(w_1u_1)=f(vu_2)$. Then the coloring $f'$ obtained from $f\vert_{E(G)}$ by uncoloring $vu_2$ is
a partial $3\Delta$-coloring of $G$, where the only uncolored edges
are $u_1v$ and $u_2v$. Again by~\eqref{s9}, we can extend $f'$ to $u_1v$ and then by the similar inequality for 
$|N^2_G[u_2v]|$, extend it to $u_2v$, a contradiction to the choice of $G$. This proves (v).
\newline

(vi) Suppose that $v$ is a poor $3_{G^*}$-vertex with $N_{G^*}(v)=\{u_1,u_2,u_3\}$ where $d_{G^*}(u_1)=2$ and
$d_{G^*}(u_2),d_{G^*}(u_3)\geq 3$. By (v), $d_{G^*}(u_2),d_{G^*}(u_3)\leq 4$. Moreover, if, say, $d_{G^*}(u_2)= 3$, then by Claim~\ref{cl: basic reducible}~(iv), $d_{G}(u_2)= 3$.
Hence $vu_1$ is a $v$-sink, yet, this contradicts~(iii).
Therefore $d_{G^*}(u_2)=d_{G^*}(u_3)= 4$. 
So, by Claim~\ref{cl: after charge}~(iii), each of $u_2$ and $u_3$ has exactly two poor neighbors in $G^*$. 
Now by (R4), each of $u_2$ and $u_3$ gives $1/2$ to $v$, while $v$ gives to $u_1$ only $1/2$ by (R5). Hence $\ch^*(v)=3+2\cdot\frac{1}{2}-1/2=7/2>3$, a contradiction to~\eqref{s6}.
\end{proof}

Claim~\ref{cl: after structure} implies that the degree of a vertex in $G^*$ must be in $\{2, 3, 4\}$. 
Moreover, since $G^*$ has neither very poor vertices nor poor $3_{G^*}$-vertices and each $4_{G^*}$-vertex is adjacent
to exactly two poor vertices,
\begin{equation}\label{s11}
\parbox{5.2in}{\em  each $2_{G^*}$-vertex is adjacent only to  $4_{G^*}$-vertices, and each $4_{G^*}$-vertex is adjacent
to exactly two
$2_{G^*}$-vertices.}
\end{equation}

The last claim that we need is:

\begin{claim}\label{cl: after structure 2}
 If $v$ is a $4_{G^*}$-vertex, then $d_G(v)=d_{G^*}(v)$.
\end{claim}

\begin{proof}
Let $v$ be a $4_{G^*}$-vertex.
By~\eqref{s11}, we may assume that $N_{G^*}(v)=\{u_1,u_2,u_3,u_4\}$, where $d_G(u_1)=d_G(u_2)=2$. So
 $$\sum_{i=1}^4 d_{G}(u_i)\leq 2(2)+ 2\Delta=2\Delta + d_{G^*}(v).$$
Hence by Claim~\ref{cl: basic reducible}~(i), $d_G(v)=d_{G^*}(v)=4$.
\end{proof}

\medskip
Now we are ready to finish the proof of Theorem~\ref{lem: 3Delta}. Since $\Mad(G^*)\leq 3$ and $G^*$ has no
$3$-regular subgraphs, it has a vertex $u$ with $d_G(u)=d_{G^*}(u)=2$. Let $N_{G}(u)=\{v,w\}$. By~\eqref{s11},
$d_{G^*}(v)=d_{G^*}(w)=4$. So by Claim~\ref{cl: after structure 2}, $d_{G}(v)=d_{G}(w)=4$.
Let $N_G(v)=\{u,v_1,v_2,v_3\}$.  By~\eqref{s11}, we may assume that $d_G(v_1)=2$.

Let $H$ be the graph obtained from $G-uv$ by adding $\Delta-4$ leaves adjacent only to $w$ so that $d_H(w)=\Delta$.
Since $d_H(u)=1$, $H$  has fewer $2^+$-vertices than $G$. Also, $\Mad(H)\leq 3$ and $H$ has no $3$-regular subgraphs.
So by the minimality of $G$, the graph $H$ has a $3\Delta$-coloring $f$. 
Since $d_{G}(w)=4$, the number of $1_H$-neighbors of the vertex $w$, including $u$, is at least $\Delta-3\geq 4$. Hence we can switch the colors of the pendant edges incident to $w$ so
that  $f(wu)\notin \{f(vv_1), f(vv_2),f(vv_3)\}$.
Thus  $f\vert_{E(G)}$ is a partial $3\Delta$-coloring of $G$, where the only uncolored edge
is $uv$. 
However, 
$$
|N^2_G[uv]|\leq d_G(w)+\sum_{i=1}^3d_G(v_i)+1\leq 4+2+2\Delta+1=2\Delta+7\leq 3\Delta,
$$
and so we can extend $f$ to $uv$, a contradiction to the choice of $G$.
This finishes the proof of Theorem~\ref{lem: 3Delta}.\newline

\noindent {\bf Acknowledgement.}
We thank the referees for careful reading and valuable comments.


\end{document}